\documentclass[a4paper, 11pt]{amsart}
\input xy
\xyoption{all}
\swapnumbers
\usepackage[top=4cm, bottom=3.5cm, left=2cm, right=2cm,marginparwidth=1.8cm, marginparsep=0.1cm]{geometry}
\usepackage[OT2, T1]{fontenc}
\usepackage[russian, english]{babel}
\usepackage{amssymb}
\usepackage{amsthm}
\usepackage{enumerate}
\usepackage[active]{srcltx}
\usepackage{amsmath}
\usepackage{nicefrac}
\usepackage{tikz-cd} 
\usepackage[colorlinks=true]{hyperref}
\usepackage{lipsum}
 \usepackage{xfrac}
\usepackage{faktor}
\usepackage{cite}
\usepackage[inline]{enumitem}

\newcommand{\CH}{\mathcal{H}}
\newcommand{\CU}{\mathcal{U}}
\newcommand{\N}{\mathbb{N}}
\newcommand{\Z}{\mathbb{Z}}
\newcommand{\wrw}{\stackrel{w}{\wr}}

\def\kP#1{\underset{ {{\tiny \mbox{ $#1$}}}}{\stackrel{\hspace{4pt} _k}{\raisebox{-.8ex}{\mbox{\fontsize{24.88}{30} $\ast$}}}}}
\def\2W#1{\underset{ {{\tiny \mbox{ $#1$}}}}{\stackrel{\hspace{4pt} _w}{\raisebox{-.8ex}{\mbox{\fontsize{24.88}{30} $\ast$}}}}}
\def\PRO#1{\underset{ {{\tiny \mbox{ $#1$}}}}{\raisebox{-.8ex}{\mbox{\fontsize{24.88}{30} $\ast$}}}}
\newcommand\vW{\stackrel{\tiny{\it{w}}}{\ast}}

\newcommand\vk{\mathrel{\stackrel{\makebox[0pt]{\mbox{\normalfont\tiny \it{k}}}}{\ast}}}

\newcommand{\sgn}{\operatorname{sign}}

\makeatletter
\DeclareRobustCommand\bigop[2][1]{%
  \mathop{\vphantom{\bigoplus}\mathpalette\bigop@{{#1}{#2}}}\slimits@
}
\newcommand{\bigop@}[2]{\bigop@@#1#2}
\newcommand{\bigop@@}[3]{%
  \vcenter{%
    \sbox\z@{$#1\bigoplus$}%
    \hbox{\resizebox{\ifx#1\displaystyle#2\fi\dimexpr\ht\z@+\dp\z@}{!}{$\m@th#3$}}%
  }%
}

\newcommand{\COMM}{\DOTSB\bigop{\ast}}

\newtheorem{proposition}{Proposition}[section]
\newtheorem{theorem}[proposition]{Theorem}
\newtheorem{lemma}[proposition]{Lemma}
\newtheorem{corollary}[proposition]{Corollary}

\newtheorem{definition}[proposition]{Definition}

 \newtheorem{convention}[proposition]{Convention}

\newtheorem{remark}[proposition]{Remark}

\newtheorem{examples}[proposition]{Examples}
\newtheorem{example}[proposition]{Example}

\numberwithin{equation}{section}

\begin{document}

\title[Permanence properties of verbal product of groups]{Permanence properties of verbal products and verbal wreath products of groups}

\author[J. Brude, R. Sasyk]{Javier Brude $^{1,2}$ \MakeLowercase{and} Rom\'an Sasyk $^{1,3}$}

\address{$^{1}$Instituto Argentino de Matem\'aticas-CONICET Saavedra 15, Piso 3 (1083), Buenos Aires, Argentina.}

\address{$^{2}$Departamento de Matem\'atica, Facultad de Ciencias Exactas, Universidad Nacional de la Plata, Argentina.}

\address{$^{3}$Departamento de Matem\'atica, Facultad de Ciencias Exactas y Naturales, Universidad de Buenos Aires, Argentina.}

\email{\textcolor[rgb]{0.00,0.00,0.84}{jbrude@mate.unlp.edu.ar}}
\email{\textcolor[rgb]{0.00,0.00,0.84}{rsasyk@dm.uba.ar}}


\date{}


\maketitle

\begin{abstract}
By means of analyzing the notion of verbal products of groups, 
we show that soficity, hyperlinearity, amenability, the Haagerup  property, the Kazhdan's property (T) and exactness 
are preserved under taking $k$-nilpotent products of groups, while being orderable is not preserved.
 We also study these properties for solvable and for Burnside products of groups.
We then show that if two discrete groups are sofic, or have the Haagerup property, 
their restricted verbal wreath product arising from nilpotent, solvable and certain Burnside products is also sofic or has the Haagerup property respectively. 
We also prove related results for hyperlinear, linear sofic and weakly sofic approximations.
 Finally, we give applications combining our work with the Shmelkin embedding to show that certain quotients of free groups are sofic or have the Haagerup property.

\end{abstract}

\section{Introduction}\label{section 1}

Given a family of groups, the direct sum and the free product provide ways of constructing new groups out of them.
Even though both operations are quite different, they share the next common properties
\begin{enumerate}
\item associativity;\label{1 intro}
\item commutativity;\label{2 intro}
\item the product contains subgroups which generate the product;\label{3 intro}
\item these subgroups are isomorphic to the original groups;\label{4 intro}
\item the intersection of a given one of these subgroups with the normal subgroup generated by the rest of these subgroups is the identity.\label{5 intro}
\end{enumerate}

In \cite{KUR53}, Kurosh asked if there were other operations on a family of groups that satisfy the above properties. This problem was solved in the affirmative in  \cite{GOL56NILP}, where, for each $k\in \N$, Golovin defined the $k$-nilpotent product of groups and proved they satisfy the aforementioned five properties. We recall that the $k$-nilpotent product is defined as follows.

\begin{definition}
Let  $\{G_i\}_{i \in \mathcal I}$ be  a family of groups indexed on a set $\mathcal I$ and consider $\mathcal F:=\COMM_{i \in \mathcal I} G_i$, the free product of the family.  Denote by $_{1}[G_i]^\mathcal F$ the cartesian subgroup of $\mathcal F$, namely the normal subgroup of $\mathcal F$ generated by commutators of the form $[g_i,g_j]$ with $g_i\in G_i$, $g_j\in G_j$ and $i\neq j$; and recursively define $_{k}[G_i]^\mathcal F:=[\mathcal F, _{k-1}[G_i]^\mathcal F]$.
The $k$-nilpotent product  of the family is the quotient group
\begin{equation*}\kP {i \in \mathcal I} G_i :=  \faktor{ \mathcal F}{ _{k}[G_i]^\mathcal F}
\end{equation*}
\end{definition}
Observe that the case $k=1$  is the direct sum of the family. Golovin showed that nilpotent products of finite groups are finite.  The articles \cite{GOL56METAB,McH}  and more recently \cite{SAS18}, further analyzed some properties of the $2$-nilpotent product of groups. 
Furthermore, for instance in  \cite{MR2450726,MR120272,MR1807651,MR2152690,MR2334753} some specific group theoretical properties like capability and the Baer invariant, of $k$-nilpotent products of finitely many finite cyclic groups were analyzed. Arguably, the study of $k$-nilpotent products of arbitrary groups, for $k\geq 3$, was missing.

In the mid fifties, Moran extended the work of Golovin in another direction by discovering a new way of defining operations in groups that fulfilled Kurosh's requirements. They are based on the notion of verbal subgroups. Recall that, roughly speaking, given a group $G$ and a subset $W$ of the free group on countably many symbols $F_{\infty}$, the verbal subgroup of $G$ for the words in $W$ is the subgroup of $G$ obtained by evaluating all the words of $W$ in the elements of $G$, (see \cite{NEU69}, or section $\S$\ref{section 2} of this article for more about verbal subgroups).  In \cite{MOR56}, Moran defined the verbal product of a family of groups  as follows.

\begin{definition}
Let  $\{G_i\}_{i \in \mathcal I}$ be  a family of groups indexed on a set $\mathcal I$ and consider $\mathcal F:=\COMM_{i \in \mathcal I} G_i$, the free product of the family. Let $W\subseteq F_{\infty}$ be a set of words and let $W(\mathcal F)$ be the corresponding verbal subgroup of $\mathcal F$. Denote by $[G_i]^\mathcal F$ the cartesian subgroup of $\mathcal F$, namely the normal subgroup of $\mathcal F$ generated by commutators of the form $[g_i,g_j]$ with $g_i\in G_i$, $g_j\in G_j$ and $i\neq j$.
The verbal product of the family is the quotient group
\begin{equation}\2W {i \in \mathcal I} G_i :=  \faktor{ \mathcal F}{W(\mathcal F)\cap [G_i]^\mathcal F}
\end{equation}
\end{definition}

 Albeit this definition is, in principle, different from the one given by Golovin; in \cite{MOR56, MOR58} Moran proved that the $k$-nilpotent product of groups is in fact an instance of verbal product of groups,  namely the verbal product obtained from a single word $n_k$,  that is recursively defined by the formulas $n_1:=[x_2,x_1]$; $n_k:=[x_{k+1},n_{k-1}]$.
 
 It is apparent that having other notions of products in groups aside from the free product and the direct sum, provide ``new'' ways of constructing groups. Hence it is of interest to study whether several structural properties and applications of the free product and the direct sum carry over to these more general operations on groups. To the best of our knowledge, aside from some articles in the late fifties and early sixties, (see, for instance, \cite{MOR58,MR0105436,MR77530,MR106941}), and the fact that products of groups were briefly mentioned in the classic books \cite{MR2109550,NEU69}, the line of research initiated by Golovin seems to have been neglected in recent years. 
 
 In \cite{SAS18}, the second named author took on the study of  $2$-nilpotent products of groups from the point of view of dynamics of groups actions and proved, among other things, that amenability, exactness, Haagerup property and Kazhdan's property (T) are preserved under taking $2$-nilpotent products of two groups. The aforementioned properties of groups are of relevance in several areas of mathematics, for instance they are at the core of the connections between  group theory and operator algebras (see, for instance,  \cite{BEK08,BRO08,CHER01,MR1308046,MR776417}, and references therein). As such, it was tempting to see if the work done in \cite{SAS18} could be extended to other nilpotent products of groups. However, since certain central sequence, that was the key ingredient in several of the proofs in \cite{SAS18}, is not present in other nilpotent products, a new idea was required to tackle this problem. 
What was even more troubling about \cite{SAS18} was that soficity and hyperlinearity, two important properties of groups that arise from the study of dynamics of group actions and that have been of very much interest in recent years due to their many applications to problems of current interest (see, for instance, \cite{MR2072092,MR2460675,Thom-ICM}), were outside the scope of the techniques deployed there. 
One of the original motivations of the present article was to address this void. Here, we solve these problems by means of analyzing the structure of the verbal products of groups. Our first main result is as follows.

\begin{theorem}\label{theorem 1 intro}
Soficity, hyperlinearity, linear soficity, weak soficity, amenability, the Haagerup approximation property, the Kazhdan's property (T) and exactness are preserved under taking $k$-nilpotent products of two countable, discrete groups.
\end{theorem}
 On the other hand, we give examples that show that the properties of being orderable or bi-orderable are not preserved under taking $k$-nilpotent products of groups, for $k\geq 2$. 
 
We also take the opportunity to start studying other families of verbal products of groups, the {\it solvable products} and the {\it Burnside products} introduced by Moran in \cite{MOR58}, 
(for their definition,  see examples \ref{examples of products}\ref{item3examples product} and \ref{examples of products}\ref{item4examples product}).
In these cases we show the following theorems.
\begin{theorem}\label{theorem 3 intro}
Soficity, hyperlinearity,  linear soficity, weak soficity, amenability, the Haagerup approximation property and exactness are preserved under taking solvable products of two countable, discrete groups. On the other hand, Kazhdan's property (T) is not preserved under taking solvable products of groups.
\end{theorem}

\begin{theorem}\label{theorem 3B intro}
Soficity, hyperlinearity,  linear soficity, weak soficity, amenability, the Haagerup approximation property, exactness and  Kazhdan's property (T) are preserved under taking $k$-Burnside products for $k=2,3,4,6$ of two finitely generated, discrete groups.
\end{theorem}

Since Moran proved that solvable products between two finite groups can be infinite (see also Proposition \ref{commutator of abelian in solvable product}), the results and examples arising from solvable products are, in general, different from the ones involving nilpotent products.

Theorems \ref{theorem 1 intro}, \ref{theorem 3 intro} and \ref{theorem 3B intro} together with the associativity of the verbal products, will allow us to prove the next corollary.
 \begin{corollary}
 \label{amenability of many products}
 If $\{G_i\}_{i\in\mathcal I}$ is a countable family of countable, discrete, sofic (hyperlinear,  linear sofic,  weakly sofic, amenable,  Haagerup, exact) groups, then the group 
$\2W {i\in\mathcal I} G_i$ is sofic, (hyperlinear,  linear sofic, weakly sofic, amenable, Haagerup, exact respectively), when the verbal product is a nilpotent  product, a solvable  product or a $k$-Burnside product for $k=2,3,4,6$.
\end{corollary}

Verbal products are suited to do a construction similar to the restricted wreath product, called {\it restricted verbal wreath products}, where the direct sum is replaced by  a verbal product,
(we refer to section $\S$\ref{section 5} for the precise  definition).
This notion was introduced by Shmelkin in \cite{MR0193131} to provide a generalization of the Magnus embedding. Since then, verbal wreath products have been used mainly as a tool in the realm of varieties of groups. Indeed, in the introduction to her classic book, \cite{NEU69}, H. Neumann wrote {\it ``Shmelkin's embedding theorem should, I believe, be made the starting point of the treatment of product varieties''}, (for further references, see for instance, \cite{MR0217158, MR879439,MR3744779}).
However, it seems that verbal wreath products had not been studied from the point of view of dynamics of group actions until \cite{MR2503171}, where they were employed, unbeknownst to their previous existance, in the classification of von Neumann algebras, (see \cite[$\S$1]{SAS18} for more details about it).

Motivated by a theorem of Cornulier, Stalder and Valette asserting that the restricted wreath product of groups with the Haagerup property has the  Haagerup property \cite{CSV1,CSV2}; in \cite{SAS18} it was showed that a similar result holds true for the restricted second nilpotent wreath product.  In this article we will explain why the same proof presented in \cite{SAS18} also serves to prove the following more general statement.

\begin{theorem}\label{theo2}
Let $G$ and $H$ be countable, discrete groups with the Haagerup property. Let $W\subseteq F_{\infty}$ be a set of words 
such that its corresponding verbal product preserves the Haagerup property. 
Then the restricted verbal wreath product between $G$ and $H$
 has the Haagerup property. In particular, restricted nilpotent wreath products, restricted solvable wreath products and restricted $k$-Burnside wreath products for $k=2,3,4,6$ between groups with the Haagerup property have the Haagerup property.
\end{theorem}

In \cite{MR3632893} Holt and Rees showed that restricted wreath products between a sofic group and an acting residually finite group is sofic. Simultaneously, in a quite technical work, Hayes and Sales relaxed the condition on the acting group and showed that soficity is preserved under taking restricted wreath products of sofic groups, \cite{HAY18}. It is still unknown if extensions of the form sofic-by-sofic are sofic, and few permanence properties of soficity are well understood, hence the importance of \cite{HAY18}.
Then, it seemed pertinent to analyze if a similar result holds true for the group extension given by the restricted second nilpotent wreath product. 
As it was explained above, the techniques presented in  \cite{SAS18} were not suitable to deal with soficity. Hence, the second and main motivation of the work carried on in this article was to study this question. Here we extend the work of \cite{HAY18} to show the following theorem.
\begin{theorem}\label{theo sofic intro}
Let $G$ and $H$ be countable, discrete, sofic groups. Let $W\subseteq F_{\infty}$ be a set of words 
such that its corresponding verbal product preserves soficity. 
Then the restricted verbal wreath product between $G$ and $H$ is sofic. In particular, restricted nilpotent wreath products, restricted solvable wreath products and restricted $k$-Burnside wreath products for $k=2,3,4,6$ between sofic groups are sofic.
\end{theorem}

Finally, as in \cite{HAY18}, we show the next variant, valid for other notions of metric approximation in groups. 

\begin{theorem}\label{theo hyper intro}
Let $G$ be a countable, discrete, hyperlinear, linear sofic or weakly sofic group and $H$ be a countable, discrete, sofic group. Let $W\subseteq F_{\infty}$ be a set of words 
such that its corresponding verbal product preserves hyperlinearity, linear soficity or weak soficity respectively. 
Then the restricted verbal wreath product between $G$ and $H$, when the acting group is $H$,  is hyperlinear, linear sofic or weakly sofic respectively. In particular, restricted nilpotent wreath products, restricted solvable wreath products and restricted $k$-Burnside wreath products for $k=2,3,4,6$ between an hyperlinear, linear sofic or weakly sofic group
 and an acting sofic group are hyperlinear, linear sofic or weakly sofic respectively.
 \end{theorem}
A straightforward application of the Shmelkin embedding together with Theorem \ref{theo2} and Theorem \ref{theo sofic intro} will yield the following result.
\begin{corollary}\label{coro Shmelkin}Let $G$ be a normal subgroup of a free group $F$. 
If $F/G$ is sofic or Haagerup, then
$F/G_k$, $F/G^{(k)}$ are sofic or Haagerup  respectively, for all $k\in\N$, where $G_k$ and $G^{(k)}$ denote the $k^{th}$ term in the lower central series and  the $k^{th}$-term in the derived series.  Moreover if $k=2,3,4,6$ then $F/G^{k}$ is sofic or Haagerup respectively, where $G^{k}$ is the $k$-Burnside subgroup of $G$.\end{corollary}
The case $F/G^\prime$ (and hence the cases $F/G^{(k)}$) has been already noticed by Hayes and Sale in \cite[Corollary 1.2]{HAY18}, and it is a consequence of their main result combined with the Magnus embedding. They also observed that it could have been deduced from an earlier theorem of P\u{a}unescu, \cite[Corollary 3.8]{MR2826401}.
To the best of our knowledge, Corollary \ref{coro Shmelkin} is the first application of the Shmelkin embedding in the context of dynamics of group actions. 

It is our hope that the work presented in this article will stimulate the study of some old beautiful constructions and results from varieties of groups as interesting objects to analyze also in the framework of geometric group theory and dynamics of group actions.

\section{Verbal subgroups and verbal products of two groups}\label{section 2}
In this section we record several properties of verbal subgroups and of verbal products of two groups. Some of the results presented here can be found scattered, at least implicitly, in \cite{GOL47,GOL56NILP,NEU69,MOR56,MOR58} and references therein. Since some of them are somewhat hard to trace and some might be obscured by lack of modern terminology, we felt compelled to present a  unified  and more or less detailed account.

\begin{convention}
In this article we adopt the convention \begin{equation*}[a,b]:=aba^{-1}b^{-1}.\end{equation*}
\end{convention}

Let $F_\infty$ be the free group on countable many letters $\{x_i\}_{i\in \mathbb{N}}$. Its elements will be called {\it words}. 
If $w\in F_\infty$, the length of $w$ will be denoted by $\ell(w)$.  A word in $n$-letters is an element of $F_\infty$ that requires at most $n$ distinct letters to be written in reduced form.  A word in $n$ letters will be noted as $w(x_{i_{1}},\dots,x_{i_{n}})$ or as $w({\bf x})$ with ${\bf x}=(x_{i_{1}},\dots,x_{i_{n}})\in F_{\infty}^{n}$.

Given a group $G$, a word in $n$ letters $w(\bf{x})$, and an element ${\bf g}=(g_1,\dots,g_n)$ in $G^{n}$, the evaluation of $w$ in ${\bf g}$ is the element 
 $w({\bf g}):=w(g_1,\dots,g_n) \in G$. The evaluation of $w$ by the elements of $G$ is the set $w(G)=\{ w({\bf g}): {\bf g} \in G^{n}\}$.\\

We recall the definition of verbal subgroup, considered first by B. H. Neumann  in \cite[$\S$5]{NEU37}.
\begin{definition}
  Let $G$ be a group. Let  $W\subseteq F_{\infty}$ be a nonempty set of words. The verbal subgroup $W(G)$ is defined as the subgroup of $G$ generated by the evaluation of all the elements of $W$ by the elements of $G$. 
   \end{definition}
   
 Observe that $S_{\rm {fin}}$, the group of finite permutations of $\{x_i\}_{i\in \mathbb{N}}$, acts on $F_{\infty}$ and if $\sigma\in S_{\rm {fin}}$, 
 then   $\sigma.w(G)=w(G)$ for all $w\in F_\infty$. 
 This implies that  for the purpose of verbal subgroups,
a word in $n$ letters can be assumed to be of the form 
 $ w({\bf x})=w(x_{1},\dots,x_{n})$.

\begin{lemma}
\label{fullyinvarianceforwords}
Let $\varphi: G \to H$ be a homomorphism of groups. Let $w \in F_{\infty}$ be a word in  $n$ letters and let ${\bf g} \in G^{n}$. Then 
$ \varphi(w({\bf g}))=w(\varphi({\bf g}))$. Hence, if $W\subseteq F_{\infty}$ is a set of words, then $\varphi(W(G))\subseteq W(H)$. 
Moreover, if $\varphi$ is surjective,  then $\varphi(W(G))=W(H)$.
 \end{lemma}
 \begin{proof} This is simply because $\varphi$ is a group homomorphism.
 \end{proof}
  Recall that if $H$ is a subgroup of a group $G$, $H$ is said to be {\it fully invariant} in $G$ if for every endomorphism $\varphi:G\to G$, $\varphi(H) \subseteq H$.
  We can now state an easy corollary to Lemma \ref{fullyinvarianceforwords} that will be used often in the remaining of this section.
\begin{corollary}\label{fully invariance of verbal subgroups}
Let $G$ be a group. Let $W\subseteq F_{\infty}$ be a set of words.
The verbal subgroup $W(G)$ is fully invariant in $G$. In particular, $W(G)$ is normal in $G$.
\end{corollary}

\begin{definition}\cite[pg. 6]{NEU69}
A set of words $W\subseteq F_{\infty}$ is  called closed if it is a fully invariant subgroup of $F_\infty$. 
The closure of $W$ is the smallest closed subgroup of $F_\infty$ that contains $W$. It is denoted by $\mathcal{C}(W)$.
\end{definition}

For the proof of the next lemma, see, for instance, \cite[12.52]{NEU69}.
\begin{lemma}\label{single w}
Let $G$ be a group. Let $W\subseteq F_{\infty}$ be a set of words. Then $\mathcal {C}(W)(G)=W(G)$. In particular, each  element of $W(G)$ can be regarded as the evaluation of a single word  $w\in \mathcal {C}(W)$ of length $n$ in an element $\bf g$ of $G^n$, for some $n\in\N$.
\end{lemma}

Let $A$ and $B$ be groups and let $A\ast B$ be the free product of them. There exists natural homomorphisms $\pi_A:A\ast B\to A$ and $\pi_B:A\ast B\to B$. By abuse of notation,  the homomorphisms 
$(\pi_{A})^{n}:(A\ast B)^{n}\to A^{n}$ and $(\pi_{B})^{n}:(A\ast B)^{n}\to A^{n}$ will also be denoted by $\pi_A$ and $\pi_B$.
\begin{definition}
Let $A$ and $B$ be groups. $[A,B]$ is the subgroup of $A\ast B$ generated by the elements of the form $[a,b]$ with $a\in A$ and $b\in B$.
\end{definition}

\begin{lemma}\label{ABC}
Let $A$ and $B$ be groups. Let  $W=\{w\}\subseteq F_{\infty}$, with $w$ a word in $n$ letters and let  $w({\bf g})\in W(A\ast B)$ be its evaluation in  ${\bf g}\in (A \ast B)^{n}$.
Then, there exists  $u \in W(A\ast B)\cap [A,B]$ such that
\begin{equation*} 
    w({\bf g}) = w(\pi_{A}({\bf g}))w(\pi_{B}({\bf g})) u
\end{equation*}
and this decomposition is unique. More precisely,  if  $  w({\bf g})=abc$, with $a\in A$, $b\in B$ and $c\in[A,B]$;
then $a= w(\pi_{A}({\bf g}))$, $b=w(\pi_{B}({\bf g}))$ and $c= u$.
\end{lemma}
\begin{proof}
We proceed by induction in the length of $w$. 
If $\ell(w)=1$, then $w$ can be supposed to be of the form $w(x_1) = x_1$ or $w(x_1) = x_1^{-1}$. In the first case, the statement of the lemma is equivalent to show that any element $a_1b_2a_3b_4\dots b_{k-1}a_{k}\in A\ast B$ with $a_i\in A$, $b_i\in B$, $b_i \neq 1$ and  $a_i\neq 1$ except perhaps for the end cases, can be written as $a_{1}a_{3}\dots a_{k}b_2b_{4}\dots b_{k-1}u$ with $u\in[A,B]$. That this holds true follows by using repetitively the identity $ba=ab[b^{-1},a^{-1}]$. The case of the word $w(x_{1})=x_{1}^{-1}$  follows from the case $w(x_{1})=x_{1}$ by means of a simple computation.

 Suppose that the lemma is valid for all the words of length less than $k$ and consider $w(x_1,\dots,x_n)=x_{i_1}^{s_1}\dots x_{i_r}^{s_r}$ with $\ell(w)=k\geq 2$ and $s_{r}\neq 0$. Write $w= w_1w_2$   with $w_1=x_{i_1}^{s_1}\dots x_{i_r}^{s_r - \sgn(s_r) }$ and  $w_2=x_{i_r}^{\sgn(s_r)}$. They have lengths $k-1$ and $1$, respectively.
By the inductive hypothesis, there exist $u', u'' \in W( A\ast B )\cap [A,B]$ such that 
 \begin{align*}
    w({\bf g})&=w_1({\bf g})w_2({\bf g})\\
    &=w_1(\pi_A({\bf g}))w_1(\pi_B({\bf g})) u' w_2(\pi_A({\bf g}))w_2(\pi_B({\bf g}))u''\\
    &=w_1(\pi_A({\bf g}))w_2(\pi_A({\bf g}))w_1(\pi_B({\bf g})) 
    u'   \Big[ ( w_1(\pi_B({\bf g})) u')^{-1}, w_2(\pi_A({\bf g}))^{-1}\Big]w_2(\pi_B({\bf g}))u'' \\
    &=w_1(\pi_A({\bf g}))w_2(\pi_A({\bf g}))w_1(\pi_B({\bf g}))w_2(\pi_B({\bf g}))u
   \\
   & =w(\pi_{A}({\bf g}))w(\pi_{B}({\bf g})) u,
    \end{align*}
    where
\begin{equation*}
u=   u'   \big[ ( w_1(\pi_B({\bf g})) u')^{-1}, w_2(\pi_A({\bf g}))^{-1}\big]
\big[   \big(u' \big[ ( w_1(\pi_B({\bf g})) u')^{-1}, w_2(\pi_A({\bf g}))^{-1}\big ] \big)^{-1}, w_2(\pi_B({\bf g}))^{-1}\big] u''.
\end{equation*}
It follows that $u\in[A,B]$. Moreover, since $u=w(\pi_{B}({\bf g}))^{-1}w(\pi_{A}({\bf g}))^{-1} w({\bf g})$, 
we have that $u\in W( A\ast B )\cap [A,B]$.
The uniqueness of the decomposition follows by taking projections.
\end{proof}
\begin{corollary}\label{When W(A) is trivial} Let $A$ and $B$ be groups. 
Let $W\subseteq F_{\infty}$ be a set of words. If $W(A)$ and $W(B)$ are both 
equal to the trivial group $\{1\}$, then $W(A\ast B)\subseteq[A,B]$.
\end{corollary}
\begin{definition}
\label{prodverbal} Let $A$ and $B$ be groups. Let $W\subseteq F_{\infty}$ be a set of words. 
The verbal product between $A$ and $B $ is 
the group
\begin{equation*}A \vW B := \faktor{A\ast B}{W(A\ast B) \cap [A,B]}\end{equation*}
The projection of $[A,B]$ onto $A \vW B$ will be denoted $[A,B]^{w}$.
\end{definition}
Verbal products were introduced by Moran in \cite{MOR56}. By definition, $A\vW B$ is equal to $B\vW A$.
If $q:A\ast B\to A\vW B$ is the quotient homomorphism, then $q|_A:A\to A\vW B$ and $q|_B:B\to A\vW B$ are injective.
For simplicity, we will also denote by $A$ and $B$ the images of $A$ and $B$ inside $A\vW B$ via the quotient homomorphism $q$.
With this identification, it is clear that the set $A\cup B\subseteq A\vW B$ generates $A\vW B$.
\begin{theorem}\cite[ Chap. 2, Theorem 1.2]{GOL56NILP}
\label{esunica}
Let $A$ and $B$ be groups. Let $W\subseteq F_{\infty}$ be a set of words. Then every element $y \in A \vW B$ has a unique writing of the form
\begin{equation*}
y= abu
\end{equation*}
with $a\in A$, $b\in B$ and $u \in [A,B]^{w}$.
 Moreover, if $y^\prime\in A \vW B$ is written as $y^\prime=a^\prime{b}^\prime u^\prime$, with $a'\in A$, $b'\in B$ and $u' \in [A,B]^{w}$, then 
$yy^\prime=aa^\prime b b^\prime \tilde{u}$, with $\tilde{u}\in[A,B]^{w}.$
\end{theorem}
\begin{proof} 
Let $q:A\ast B \to A \vW B$ be the quotient homomorphism. There exists $g\in A\ast B$ such that $y=q(g)$. Lemma \ref{ABC} applied to the word $w=x_1$ says that $g$ can be written as 
$g=abu$ with $a\in A$, $b\in B$ and $u\in[A,B]$. Then $y=q(g)=q(a)q(b)q(u)=abq(u)$, with $q(u)\in [A,B]^{w}$. 
The uniqueness of the decomposition follows by taking projections. 

The same procedure shows that $y^\prime=q(g^\prime)=q(a^\prime{b}^\prime u^\prime)$. 
 Lemma \ref{ABC} implies that $yy^\prime=q(gg^\prime)=q(\pi_{A}(gg^\prime)\pi_{B}(gg^\prime)\tilde{u})=q(aa^\prime bb^\prime\tilde{u})=aa^\prime bb^\prime q(\tilde{u})$ with $q(\tilde{u})\in [A,B]^{w}$.
\end{proof}

By means of Theorem \ref{esunica}, it is easy to deduce that $A\cap \overline{B}^{^{A \vW B}}=\{1\}$ and $B\cap \overline{A}^{^{A \vW B}}=\{1\}$, 
where the symbol 
 $-^{A \vW B}$ denotes the normal closure inside $A \vW B$.
\begin{proposition}\label{trivial intersection}
Let $A$ and $B$ be groups. Let $W\subseteq F_{\infty}$ be a set of words. Then 
\begin{center}
$W(A \vW B) \cap [A,B]^{w} = \{1\}$.
\end{center}
\end{proposition}

\begin{proof}
Let $q:A\ast B \to A \vW B$ be the quotient homomorphism. 
By Lemma \ref{fullyinvarianceforwords}, $W(A\vW B) = W(q(A \ast B))= q(W(A \ast B))$. 
Hence if $y\in W(A \vW B) \cap [A,B]^{w}$, then $y=q(g)$ with $g\in W(A \ast B)$ and $y=q(\tilde{g})$ with $\tilde{g}\in [A,B]$. 
It follows that $g\tilde{g}^{-1} \in \ker (q)=W(A \ast B) \cap [A,B]$ and hence $g\in [A,B]$ and 
$\tilde{g}\in W(A\ast B)$. We conclude that $y=1$ in $W(A \vW B)$.
\end{proof}

\begin{lemma}\label{commute with word in A}
Let $A$ and $B$ be groups. Let $W\subseteq F_{\infty}$ be a set of words. Let $a \in W(A)\subseteq  A\vW B$. Then $[a,B]= 1$  in $A\vW B$ and $[a,[A,B]^{w}] =1$ in $A\vW B$.
\end{lemma}

\begin{proof}
By  Lemma \ref{single w}, it is enough to prove the statement for $a$ 
of the form $a=w({\bf a})$ where $w\in\mathcal{C}(W)$ is a word in $n$ letters and ${\bf a} \in A^{n}$ for some $n\in\N$.
In order to show the first part of the lemma, let $b \in B\subseteq A\vW B$. Then
$[a,b]= w({\bf a})bw({\bf a})^{-1}b^{-1} \in [A,B]^{w}$.
 Corollary \ref{fully invariance of verbal subgroups}, implies that $bw({\bf a})^{-1}b^{-1} \in W(A\vW B)$, and then $[a,b] \in W(A\vW B)$. Thus $[a,b] \in W(A \vW B) \cap [A,B]^{w}=\{1\}$, by Proposition \ref{trivial intersection}.
 
In order to show the second part of the lemma, let $c \in [A,B]^{w}$.  Corollary \ref{fully invariance of verbal subgroups}, implies that $ca^{-1}c^{-1} \in W(A\vW B)$, then $[a,c] \in W(A\vW B)$, and as $aca^{-1} \in [A,B]^{w}$, then $[a,c]\in [A,B]^{w}$. Hence $[a,c] \in W(A\vW B)\cap [A,B]^{w} = \{1\}$, by Proposition \ref{trivial intersection}.
\end{proof}
The following result can be found in \cite[Corollary 4.2.2]{MOR56}. 
We provide a short a proof of it.

\begin{proposition}\label{prod}
Let $A$ and $B$ be groups. Let $W\subseteq F_{\infty}$ be a set of words.  
The function  \begin{equation*}\Psi:W(A) \times W(B)\to W( A \vW B)\end{equation*} defined by $\Psi(a,b):=ab$ is a group isomorphism.
\end{proposition}
\begin{proof}
By Lemma  \ref{commute with word in A},
\begin{equation*}\Psi(a,b)\Psi(a^\prime,b^\prime)= aba^\prime b^\prime
=aa^\prime [a^{\prime -1},b] b b^\prime=aa^\prime b  b^\prime=\Psi((a,b)(a^\prime,b^\prime)).
\end{equation*}
Hence  $\Psi$ is a group homomorphism. Theorem \ref{esunica} implies that $\Psi$ is injective.
In order  to show that it is surjective,
take $y \in W(A\vW B)$. Then $y=q(g)$ with $g\in W(A \ast B)$. By Lemma \ref{single w}, $g$ is of the form 
$g=w({\bf g})$ where $w\in\mathcal{C}(W)$ is a word in $n$ letters and ${\bf g} \in (A\ast B)^{n}$, for some $n\in\N$.
By Lemma \ref{ABC}, $w({\bf g})=w(\pi_{A}({\bf g}))w(\pi_{B}({\bf g}))u$ with $u\in W(A\ast B)\cap[A,B]$.
Hence
 \begin{align*}y&=q(w({\bf g}))= q(w(\pi_{A}({\bf g}))w(\pi_{B}({\bf g}))u)=q(w(\pi_{A}({\bf g})))q(w(\pi_{B}({\bf g})))\\
&=\Psi(w(\pi_{A}({\bf g})),w(\pi_{B}({\bf g}))).\qedhere
\end{align*}
\end{proof}
\begin{lemma}
\label{igualdadintersec} Let $A$ and $B$ be groups.  Let $W\subseteq F_{\infty}$ be a set of words. Let $M$ be a normal subgroup of $A$ and let $N$ be a normal subgroup of $B$.
Denote  by $\overline{MN}^{^{A \ast B}}$  the normal closure of the group generated by $M$ and $N$ inside the free product $A \ast B$.
Let $\phi:A \ast B \to \nicefrac{A}{M} \ast \nicefrac{B}{N}$ be the natural homomorphism. Then 
\begin{enumerate}
\item  \label{uno}
 $\ker(\phi)=\overline{MN}^{^{A \ast B}}$,
\item \label{dos} 
$\phi(W(A\ast B) \cap [A,B]) = W(\phi(A\ast B)) \cap [\phi(A),\phi(B)]$.
\end{enumerate}
\end{lemma}
\begin{proof} 
Item \eqref{uno} is well known. For the sake of completeness, we include a proof here.
Since $M,N\subseteq\ker(\phi)$, then $\overline{MN}^{^{A \ast B}} \subseteq \ker(\phi)$. So, it is enough to prove the reverse  inclusion. 
We proceed by induction on the number of factors of a reduced word $g\in\ker(\phi)\subseteq A\ast B$.
If $g$ has at most one factor, the result is immediately true.
 Suppose that every reduced word in $\ker(\phi)$ with at most $n-1$ factors belongs to $\overline{MN}^{^{A \ast B}}$.
Let $g=g_1g_2\dots g_n$  be a reduced word with $n\geq 2$ factors in $A\ast B$ and suppose that $g\in \ker(\phi)$. 
 If $\phi(g_i)\neq 1$  for every $1\leq i \leq n$, then $\phi(g)$ is a reduced word with $n\geq 2$ factors in $\nicefrac{A}{M} \ast \nicefrac{B}{N}$,  contradicting that $\phi(g)=1$. 
 Then there exists $k$ such that $\phi(g_k) = 1$.
 Define $\tilde{g}:= g_1 \dots g_{k-1} $ and  $\hat{g}:= \tilde{g} g_{k+1} \dots g_n$. 
 We have that $\phi(\hat{g})=1$.
 By the inductive hypothesis, $\hat{g} \in \overline{MN}^{^{A \ast B}}$.
 Observe that since  $g_k$ is either in $M$ or in $N$, then $\tilde{g}g_k \tilde{g}^{-1}\in \overline{MN}^{^{A \ast B}}$. 
Writing $g =(\tilde{g} g_k \tilde{g}^{-1}) \hat{g}$,
it follows that $g$ belongs to $\overline{MN}^{^{A \ast B}}$.

Let us proceed to prove \eqref{dos}. For one of the inclusions it is enough to observe that
\begin{equation*}\phi(W(A\ast B) \cap [A,B]) \subseteq \phi(W(A\ast B)) \cap\phi( [A,B]) = W(\phi(A\ast B)) \cap [\phi(A),\phi(B)],\end{equation*} where the last equality is valid due to Corollary \ref{fully invariance of verbal subgroups} and the fact that, since $\phi$ is surjective, then $[\phi(A),\phi(B)]= \phi([A,B])$.

 In  order to prove the reverse inclusion, let $y \in W(\phi(A\ast B)) \cap [\phi(A),\phi(B)]$.
 On the one hand, $y \in W(\phi(A\ast B))=\phi(W(A \ast B))$. Then by Lemma \ref{single w}, $y$ is of the form 
$y=\phi(w({\bf g}))$ where $w\in\mathcal{C}(W)$ is a word in $n$ letters and ${\bf g} \in (A\ast B)^{n}$, for some $n\in\N$.
By Lemma \ref{ABC}, $y=\phi(w({\bf g}))=\phi(w(\pi_{A}({\bf g})))\phi(w(\pi_{B}({\bf g})))\phi(u)$ with $\phi(u)\in \phi(W(A\ast B)\cap[A,B])\subseteq [\phi(A),\phi(B)]=[A/M,B/N]$.
On the other hand $y \in [A/M,B/N]$. Then the uniqueness assertion of Lemma \ref{ABC} applied to the word $w=x_1$ and the group $\nicefrac{A}{M} \ast \nicefrac{B}{N}$, implies that $\phi(w(\pi_{A}({\bf g})))=\phi(w(\pi_{B}({\bf g})))=1$. Hence $y= \phi(u) \in \phi (W(A\ast B) \cap [A,B])$.
\end{proof}

\begin{theorem}\label{quotient}
Let $A$ and $B$ be groups.  Let $W\subseteq F_\infty$ be a set of words. Let $M$ be a normal subgroup of $A$ and let $N$ be a normal subgroup of $B$. 
 Then
\begin{equation*}\faktor{A\vW B}{\overline{MN}^{^{A \vW B}}} \cong \faktor{A}{M} \vW \faktor{B}{N}\end{equation*}
where $\overline{MN}^{^{A \vW B}}$ is the normal closure of the group generated by $M$ and $N$ inside $A \vW B$.
\end{theorem}

\begin{proof}
We consider the following diagram
\[
\begin{tikzcd}
A\ast B \arrow{r}{\phi} \arrow[swap]{d}{q} & \faktor{A}{M}\ast\faktor{B}{N} \arrow{d}{\tilde{q}} \\
A \vW B \arrow[dashed]{r}{\Phi} & \faktor{A}{M}\vW\faktor{B}{N}
\end{tikzcd}
\]
$\Phi$ is a well defined  homomorphism because
\begin{align}\label{cuenta}
\phi(\ker(q)) &=\phi(W(A\ast B) \cap [A,B]) 
\nonumber\\ 
&= W(\phi(A\ast B)) \cap [\phi(A),\phi(B)] 
\nonumber
\\&= W\Big(\faktor{A}{M}\ast\faktor{B}{N}\Big)\cap \Big[\faktor{A}{M},\faktor{B}{N}\Big]
\nonumber\\
&=\ker(\tilde{q}).
\end{align}
where the second equality is due to Lemma \ref{igualdadintersec}\eqref{dos}. Moreover, since  $\phi$ and $\tilde{q}$ are surjective, $\Phi$ is surjective.
On the other hand,
\begin{align*}
    \ker(\Phi)&=q(\phi^{-1}(\ker(\tilde{q})))=q(\phi^{-1}(\phi (\ker(q))))  = q (\ker(q) \ker(\phi )) 
    = q (\ker( \phi)) \\
    &= q( \overline{MN}^{A\ast B})
    = \overline{MN}^{A\vW B},
    \end{align*}
    where the third equality is due to equation \eqref{cuenta}, the second to last equality is due Lemma \ref{igualdadintersec}\eqref{uno}, and the last equality is because $q$ is surjective.
\end{proof}

\begin{corollary}\label{coro on short exact sequence}  Let $A$ and $B$ be groups.  Let $W\subseteq F_\infty$ be a set of words. 
Then $ \faktor{A \vW B}{W(A \vW B)}$ is isomorphic to  
$\faktor{A}{W(A)} \vW \faktor{B}{W(B)}$ and this gives the short exact sequence
\begin{equation}\label{sec}
 1 \to W(A)\times W(B) \to A \vW B \to   \faktor{A}{W(A)} \vW \faktor{B}{W(B)}\to 1.
\end{equation}
Moreover the subgroup $[A,B]^w$ of $A\vW B$ is isomorphic to 
the subgroup $\Big[\faktor{A}{W(A)}, \faktor{B}{W(B)}\Big]^w$ of $\,\,\faktor{A}{W(A)} \vW \faktor{B}{W(B)}$.
\end{corollary}
\begin{proof} By Proposition \ref{prod},  $W(A)\times W(B)$ is isomorphic to  $W(A)W(B)$, where we regard  $W(A)W(B)$ as a subgroup of $A\vW B$ via the isomorphism $\Psi$, and this is equal to $W( A \vW B)$.
Since $W( A \vW B)$  is a verbal subgroup of  $A\vW B$, it is normal in $A\vW B$. Hence, if in Theorem \ref{quotient} we take $M=W(A)$ and $N=W(B)$, we get that $\ker(\Phi)=W(A\vW B)$ and then $ \faktor{A \vW B}{W(A \vW B)} \cong \faktor{A}{W(A)} \vW \faktor{B}{W(B)}$.

Moreover, the commutative diagram in the proof of Theorem \ref{quotient}, gives that 
\begin{equation*}\Phi([A,B]^w)=\Big[\faktor{A}{W(A)}, \faktor{B}{W(B)}\Big]^w.
\end{equation*} 
By Proposition \ref{trivial intersection} $W(A \vW B) \cap [A,B]^{w} = \{1\}$.  It follows that 
$\Phi|_{[A,B]^{w}}$ is injective.
\end{proof}
The next result generalizes \cite[Corollary 2.18]{SAS18}, where it was done in the case of the second nilpotent product, but by different methods.
\begin{corollary}  Let $A$ and $B$ be groups.  Let $W\subseteq F_\infty$ be a set of words. If $W(A)=A$, it follows that
$A\vW B$ is isomorphic to $A\times B$.
\end{corollary}
\begin{proof}
By Corollary \ref{coro on short exact sequence},  $[A,B]^w=\{1\}$. The result then follows from Theorem \ref{esunica}.
\end{proof}

We end this section with some examples of verbal subgroups and of verbal products of groups, (see, also, \cite[$\S$5]{MOR56}).
\begin{examples}[of verbal subgroups] \label{examples of verbal groups} Given a group $G$, the verbal subgroup given by
\begin{enumerate}[label=(\roman*)]
\item the empty word is the identity of $G$;
\item \label{item1examples} the word $n_1:=[x_2,x_1]$  is  the commutator subgroup of $G$;
\item \label{item2examples} the words  $n_k:=[x_{k+1},n_{k-1}]$ with $k\in \mathbb N_{\geq 2}$ recursively yield the lower central series of $G$;
\item \label{item3examples} the words
\begin{align*}
s_1(x_1,x_2)&:=[x_1,x_2],\\
  s_k(x_1,\ldots,x_{2^k})
& :=
[s_{k-1}(x_{1},\ldots,x_{2^{k-1}}),s_{k-1}(x_{2^{k-1}+1},\ldots,x_{2^{k}})],
\end{align*}
recursively yield the derived series of $G$;
\item \label{item4examples} the word $x_1^{k}$ gives the $k$-Burnside's subgroup, namely, the group generated by the  $k^{\rm{th}}$ power of elements of $G$.
\end{enumerate}
\end{examples}

\begin{examples} [of verbal products]\label{examples of products} The words in the examples \ref{examples of verbal groups} give the following verbal products.
\begin{enumerate}[label=(\roman*)]
    \item If $W=\{\}$, (\text{the empty word}), then $A\vW B = A\ast B$;
    \item \label{item1examples product}if $W=\{n_1\}$, with $n_1$ defined in \ref{examples of verbal groups}\ref{item1examples}, then $A \vW B = A \times B$;
    \item  \label{item2examples product} if $W=\{n_k\}$, with $n_k$ defined in \ref{examples of verbal groups}\ref{item2examples}, then $A \vW B$ coincides with the $k$-nilpotent product of groups, $A \vk B$, first studied by Golovin; 
    \item  \label{item3examples product}if $W=\{s_k\}$, with $s_k$ as in \ref{examples of verbal groups}\ref{item3examples}, then $A\vW B$
    is called the $k$-solvable product;
    \item \label{item4examples product} if $W = \{x_{1}^k\}$, the product is called the  $k$-Burnside product.
\end{enumerate}
\end{examples}

Recall that given a variety of groups $\mathfrak B$ with laws $W\subseteq F_\infty$, 
{\it the free group of rank $n$ in the variety } $\mathfrak B$ is the group $F_n/W(F_n)$ (see, for instance, \cite[Chapter 1.4]{NEU69}).
Observe that if $W(A)=W(B)=1$ then 
by Corollary \ref{When W(A) is trivial} we have that $A\vW B=A\ast B/W(A\ast B)$. 
\begin{examples} \label{concrete examples of products}  
We present several verbal products as free groups in varieties of groups.
\begin{enumerate}[label=(\roman*)]
    \item If $W=\{n_k\}$, with $n_k$ defined in \ref{examples of verbal groups}\ref{item2examples}, then $\Z \vW \Z$ is
    the free nilpotent group of class $k$ and rank $2$;
    \item If $W=\{s_k\}$, with $s_k$ as in \ref{examples of verbal groups}\ref{item3examples}, then $\Z\vW \Z$
    is the free solvable group of class $k$ and rank $2$;
    \item \label{item burnside concrete examples of products}   if $W = \{x_{1}^k\}$, then $\Z/k\Z \vW \Z/k\Z=\frac{\Z/k\Z\, \ast\, \Z/k\Z}{W(\Z/k\Z\, \ast\, \Z/k\Z)}\cong F_2/W(F_2)$ is
    the free Burnside group of class $k$ and rank $2$, namely $B(2,k)$.
\end{enumerate}
\end{examples}

\section{Permanence properties of verbal products of two groups}\label{section 3}
With the tools developed in the previous section, here we will prove Theorems \ref{theorem 1 intro}, \ref{theorem 3 intro}, and \ref{theorem 3B intro} from the introduction.
Before doing that, we first recall that
in \cite[Theorem 6.11]{GOL56NILP} Golovin showed that the $k$-nilpotent product of nilpotent groups is nilpotent, 
while in \cite[Theorem 9.2, Theorem 10.2]{MOR58} Moran showed the analogous results for solvable and Burnside products.
These results are straightforward corollaries of Proposition \ref{prod}. 
We record a more precise description in the next proposition.
\begin{proposition} \label{preserves basic}
Let $A$ and $B$ be groups. Let $W\subseteq F_{\infty}$ be a set of words.
\begin{enumerate} 
\item \label{preserves nil} 
If $A$ and $B$ are nilpotent of orders $m$ and $n$ respectively, and $W=\{n_k\}$ as in example \ref{examples of products}\ref {item2examples product},
Then $A\vW B$ is a nilpotent group of order at most $\max\{k,m,n\}$.

\item \label{preserves sol} 
If $A$ and $B$ are solvable of derived length $m$ and $n$ respectively, and $W=\{s_k\}$ as in example \ref{examples of products}\ref {item3examples product},
Then $A\vW B$ is a solvable group of derived length at most $\max\{k,m,n\}$.
\item \label{preserves Burnside}
If $A$ and $B$ have exponents $m$ and $n$ respectively, and  $W=\{x_1^k\}$ as in example \ref{examples of products}\ref {item4examples product},
then $A\vW B$ has exponent  ${\rm LCM}(m,n,k)$.
\end{enumerate}
\end{proposition}

As it was mentioned in the introduction, the second named author showed that the second nilpotent products of groups preserves several  group theoretical properties that are of interest in representation theory, dynamics of group actions and operator algebras.
The key fact in most of the proofs given in \cite[Proposition 3.1]{SAS18} was the short exact sequence 
\begin{equation} \label{Sec 2 nil}
1 \to \nicefrac{A}{[A,A]}\otimes \nicefrac{B}{[B,B]} \to A \stackrel{_2}{\ast} B \to A \times B \to 1.
\end{equation}
This exact sequence even allowed to compute the order of $ A \stackrel{_2}{\ast} B$. However, the exact sequence relied on some unique features of the second nilpotent product that are not present in other verbal products.

For the sake of clarity and for future reference, we state the next more precise variant of Theorems \ref{theorem 1 intro} and \ref{theorem 3 intro}.
\begin{proposition}\label{StabProp}
Let $A$ and $B$ be countable, discrete, groups. Let $W\subseteq F_\infty$ be a set of words as in examples \ref{examples of products}\ref {item2examples product} and  \ref{examples of products}\ref {item3examples product}  that define the nilpotent and solvable products. Then $A \vW B$ has one of the following properties:
\begin{enumerate}[label=(\arabic*)]
\item\label{sofic} sofic;
\item\label{hyper} hyperlinear;
\item\label{weak sofic} weak sofic;
\item\label{linear sofic} linear sofic;
\item\label{amen} amenable;
\item\label{haag} Haagerup approximation property;
\item \label{exact} exact (or boundary amenable, or satisfies property A of Yu);
\end{enumerate}
if and only if $A$ and $B$ have the same property.
\end{proposition}
\begin{proof}
We refer  to \cite{BEK08,BRO08,CHER01,MR2460675,MR3592512,MR2455513} for the definitions and thorough treatments of the properties of groups stated above, (see also section $\S$\ref{section 5} for a brief discussion on each of the first four items).
 Since $A$ and $B$ are subgroups of $A \vW B$ and the properties  from \ref{sofic}
 to \ref{exact} are inherited by subgroups, it follows that if $A \vW B$ satisfies one of these seven properties, then both $A$ and $B$ must also satisfy it. 
We are now left to show the reverse implications.
To that end, we consider the short exact sequence
\begin{equation*}
 1 \to W(A)\times W(B) \to A \vW B \to   \faktor{A}{W(A)} \vW \faktor{B}{W(B)}\to 1.
\end{equation*}

Observe that  $\faktor{A}{W(A)}$ and  $\faktor{B}{W(B)}$ are nilpotent when $W=\{n_k\}$ and solvable when  $W=\{s_k\}$. Hence by Proposition  \ref{preserves basic}, $\faktor{A}{W(A)} \vW \faktor{B}{W(B)}$ is nilpotent when $W=\{n_k\}$ and solvable when
 $W=\{s_k\}$. It follows that in both cases, $\faktor{A}{W(A)} \vW \faktor{B}{W(B)}$ is an amenable group.

In order to prove  \ref{sofic}, recall that that $W(A)\times W(B)$ is  sofic, since it is a product of subgroups of sofic groups. Then, $A \vW B$ is a sofic-by-amenable extension and by \cite[Theorem 1]{ELE03}, it is sofic.

In order to prove  \ref{hyper}, recall that $W(A)\times W(B)$ is hyperlinear, since it is a product of subgroups of hyperlinear groups. Then, $A \vW B$ is a hyperlinear-by-amenable extension and by a suitable adaptation of \cite[Theorem 1]{ELE03}, (also see \cite[Theorem B]{MR3934790}), it is hyperlinear. The proofs of \ref{weak sofic} and \ref{linear sofic} are almost identical to this, one has to use \cite[Theorem 5.1]{MR3632893} and \cite[Theorem 9.3]{MR3592512}.

In order to prove \ref{amen}, recall that $W(A)\times W(B)$ is amenable since it is a product of subgroups of amenable groups.
Therefore $A\vW B$ is amenable because it is an extension of amenable groups (see \cite[Theorem G.2.2]{BEK08}).

In order to prove \ref{haag}, recall that the Haagerup Property is preserved by taking subgroups and finite direct products,
 thus  the group $W(A)\times W(B)$ has the Haagerup Property. Then  $A \vW B$ is a Haagerup-by-amenable extension, 
 and by \cite[Example 6.1.6]{CHER01}, it has the Haagerup Property.

In order to prove \ref{exact}, recall that amenable groups are exact, and that subgroups and extensions of exact groups are exact \cite[Proposition 5.1.11]{BRO08}.
\end{proof}

\begin{remark}
In \cite{SAS18}, it was proved that the second nilpotent product of groups satisfies items \ref{amen}, \ref{haag} and \ref{exact} of Proposition \ref{StabProp}.
However, since it is not known whether  abelian-by-sofic extensions are sofic, the short exact sequence  \eqref{Sec 2 nil} can not be used to show that the second nilpotent product of groups satisfies \ref{sofic}. Addressing this issue was one of the first motivations to carry on the work  presented in this article.
\end{remark}
\begin{proposition}\label{StabProp T}
Let $A$ and $B$ be countable groups. Let $W=\{n_k\}$ where $n_k$ is as in example \ref{examples of products}\ref {item2examples product}. Then $A \vW B$ has Kazhdan's property (T) if and only if $A$ and $B$ have it.
\end{proposition}
\begin{proof}
By the discussion ensuing Theorem \ref{esunica}, (or by Theorem \ref{quotient}), we have that 
$\nicefrac{A\vW B}{\overline{B}^{^{A \vW B}}} \cong  A$
 and $\nicefrac{A\vW B}{\overline{A}^{^{A \vW B}}} \cong  B$.
Since Property (T) is inherited by quotients (see, for instance,  \cite[Theorem 1.3.4]{BEK08}) it follows that if $A\vW B$ has property (T), then $A$ and $B$ have property (T).

For the reverse implication, suppose that $A$ and $B$ have property (T). Then the quotients $\nicefrac{A}{W(A)}$ and $\nicefrac{B}{W(B)}$ have property (T). Moreover, they are $k$-nilpotent, so they are amenable. Hence they must be finite (see \cite[Theorem 1.1.6]{BEK08}). As $k$-nilpotent products between two finite groups is a finite group (see  \cite[Theorem, $\S$ 6]{GOL56NILP}), then  $\faktor{A}{W(A)} \vW \faktor{B}{W(B)}$ is finite.
 On the other hand,  since $W(A)$ is a normal subgroup of $A$ and $\faktor{A}{W(A)}$ is finite, 
 W(A) has finite index in $A$, and then it  
 has property (T) (see, for instance, \cite[Theorem 1.7.1]{BEK08}). The same happens with $W(B)$. It follows that $W(A)\times W(B)$ has property (T). 
 Then, both ends of \eqref{sec} have property (T). Hence, by \cite[Theorem 1.7.6]{BEK08}, $A\vW B$  has property (T).
\end{proof}

We will next see that, in contrast to nilpotent products, solvable products do not preserve property  (T). The reason behind it is that, in general, solvable products of finite groups are not necessarily finite. This last fact has already been noticed by Moran in \cite[Theorem 9.4 and Corollary 9.4.1]{MOR58}.
We provide a slightly more precise description in the next proposition.
\begin{proposition}\label{commutator of abelian in solvable product}
Let $A$ and $B$ be abelian groups.
 Let $W=\{s_k\}$, with $s_k$ as in example \ref{examples of products} \ref{item3examples product}, and $k\geq 2$.
  Then $[A,B]^{w}\subseteq A\vW B$ 
is the free solvable group of derived length $k-1$ on the set $\{[a,b]: a\in A\setminus \{1\}, b\in B\setminus \{1\} \}$.
\end{proposition}
\begin{proof} 
If $A$ and $B$ are groups, then 
$[A,B]$ is a free subgroup of $A\ast B$ in the generators $\{[a,b]: a\in A\setminus \{1\}, b\in B\setminus \{1\} \}$. Moreover it is normal in $A\ast B$, and
 if $A$ and $B$ are abelian, then $[A\ast B,A\ast B]=[A,B]$. (For a proof of this elementary fact, see, for instance,  \cite[Section 1.3, Proposition 4]{serreTRees}).

Then, if $A$ and $B$ are abelian, it follows by induction that for every $k\geq 2$, $s_k(A\ast B)=s_{k-1}([A,B])\subseteq [A,B]$.
Hence, by  Definition \ref{prodverbal}, 
$A \vW B = \faktor{A\ast B}{s_{k-1}([A,B])}$. Hence $[A,B]^{w}$, the projection of $[A,B]$ onto $A \vW B$, is isomorphic to  $[A,B]/{s_{k-1}([A,B])}$. 
Since $[A,B]$ is a free group, this is, by antonomasia, the free solvable group of derived length $k-1$ on the set $\{[a,b]: a\in A\setminus \{1\}, b\in B\setminus \{1\} \}$.
\end{proof}
\begin{corollary}
  Let $W=\{s_k\}$, with $s_k$ as in example \ref{examples of products} \ref{item3examples product}, and $k\geq 2$.
 Then $\nicefrac{\mathbb Z}{n\mathbb Z}\vW \nicefrac{\mathbb Z}{n\mathbb Z}$, is an infinite solvable group.
  In particular, the property (T) of Kazhdan is not preserved under taking solvable products of groups. 
\end{corollary}

Since the direct sum and free product of two orderable groups is an orderable group, (see, for instance, \cite{CLA16,DNR14}), it is natural to ask if the same remains true for other verbal products. 
The following examples show that this is not the case.
\begin{example}
Let $G=\pi_1(\text{Klein bottle})\cong \langle a,b: aba^{-1}b=1\rangle$. It is easy to show that it is a left-orderable group but it is not bi-orderable, (see, for instance, \cite[Example 1.9]{CLA16} or \cite[$\S$1.3.4]{DNR14}).
Its abelianization is $\rm{H_1}(\text{Klein bottle})\cong \Z\times\frac{\Z}{2\Z}$.
Hence, the group $G\stackrel{_2}{\ast} G$ is not left-orderable, since, by the exact sequence \eqref{Sec 2 nil}, it contains a subgroup isomorphic to $\big(\Z\times\frac{\Z}{2\Z}\big)\otimes \big(\Z\times\frac{\Z}{2\Z}\big)$, thus it has torsion.
One can show that for any $k\geq 2$, the $k$-nilpotent product $G\stackrel{_k}{\ast} G$ is not left-orderable.
\end{example}
Torsion free nilpotent groups are bi-orderable, (see, for instance, \cite[$\S$1.2.1]{DNR14}). One might ask if nilpotent products of torsion free nilpotent groups are bi-orderable. In light of Proposition \ref{preserves basic} \eqref{preserves nil}, this is equivalent to ask if nilpotent products of torsion free nilpotent groups are torsion free. The answer is no.
The following example is taken from \cite[$\S$1.2.2]{DNR14}.

\begin{example}
Let
$G=\Bigg\{\begin{pmatrix} 1 && 2a &&c\\0 &&1 && 2b\\0&& 0&& 1\end{pmatrix}:\,\,a,b,c\in\mathbb{Z} \Bigg\}$. It is a torsion free group and nilpotent of class $2$.
Its commutator subgroup is $[G,G]=\Bigg\{\begin{pmatrix} 1 && 0 &&4c\\0 &&1 && 0\\0&& 0&& 1\end{pmatrix}:\,\,c\in\mathbb{Z} \Bigg\}$. Its abelianization is isomorphic to 
$\Z^2\times\frac{\Z}{4\Z}$, so it has torsion. Hence, by the exact sequence \eqref{Sec 2 nil}, the $2$-nil group $G\stackrel{_2}{\ast} G$ contains a subgroup isomorphic to
 $\big(\Z^2\times\frac{\Z}{4\Z}\big)\otimes \big(\Z^2\times\frac{\Z}{4\Z}\big)$, and hence it has torsion. One can show that for any $k\geq 2$, the 
 $k$-nilpotent product $G\stackrel{_k}{\ast} G$ has torsion.
 \end{example}
 
 A question remains: is it true that solvable products of torsion free solvable groups are bi-orderable? 
 Perhaps, in light of Proposition \ref{commutator of abelian in solvable product}, the answer is yes.
 
  \begin{remark}A similar study could be attempted to do for Burnside products. Here the situation is far more delicate.  Adyan proved in \cite{MR682486} that the free Burnside groups $B(m, n)$ are non amenable whenever $m \geq 2$, and $n\geq 665$ odd. Hence, using example \ref{concrete examples of products}\ref{item burnside concrete examples of products} 
  when $W=\{x^{665}\}$, we have that
    $\Z/665\Z\vW \Z/665\Z\cong B(2,665)$, so Burnside products do not preserve amenability. However, 
  Burnside, Sanov and Hall showed that $B(m,n)$ is finite when $n=2,3,4,6$. Then, for $k$-Burnside products with $k=2,3,4,6$ one could get results like the one stated above by means of the same  techniques employed in this article. Indeed,  if $A$ and $B$ are finitely generated, then  Proposition \ref{preserves basic}\eqref{preserves Burnside}
  implies that when $W=\{x_1^k\}$, the group $\faktor{A}{W(A)} \vW \faktor{B}{W(B)}$ is a finitely generated and of exponent $k$. Hence it is finite when $k=2,3,4,6$. Then exactly the same proofs of Proposition \ref{StabProp} and Proposition \ref{StabProp T} serve to show 
  Theorem \ref{theorem 3B intro} from the introduction.
\end{remark}

\section{Verbal products of arbitrarily many groups}

The purpose of this section is to prove Corollary \ref{amenability of many products}, and more importantly, to set up the premises needed to prove
 the results regarding restricted verbal wreath products  discussed in the introduction. To that end, we start with two definitions.
 \begin{definition}\label{mixed commutator}
Let $\{G_i\}_{i\in\mathcal I}$ be a family of subgroups of a group $\mathcal G$. the symbol $[G_i]^{\mathcal G}$ 
denotes the normal subgroup of $\mathcal G$ generated 
by the elements of the form $[g_i,g_j]$ with $g_i\in G_i, g_j\in G_j$ with  $i\neq j$.
\end{definition}
In the case when $\mathcal G$ equals the free product of the groups, $[G_i]^{\mathcal G}$ is called the cartesian subgroup of $\mathcal {G}$, \cite[18.17]{NEU69}. When all the $G_i$ are equal to a fixed group $G$, we will just write $[G]^{\mathcal G}$.

\begin{definition}\cite[$\S$4]{MOR56}
\label{prodverbal of many groups}
Let  $\{G_i\}_{i \in \mathcal I}$ be  a family of groups indexed on a set $\mathcal I$ and consider $\mathcal F:=\COMM_{i \in \mathcal I} G_i$ the free product of the family. Let $W\subseteq F_\infty$ be a set of words and let $W(\mathcal F)$ be the corresponding verbal subgroup of $\mathcal F$.
The verbal product of the family is the quotient group
\begin{equation*}\2W {i \in \mathcal I} G_i :=  \faktor{ \mathcal F}{W(\mathcal F)\cap [G_i]^\mathcal F}\end{equation*}
\end{definition}
As in section $\S$\ref{section 2}, the letter $q$ will denote the quotient homomorphism.
The proof of the next  generalization of Theorem  \ref{esunica} is left to the reader, (see, for instance, \cite[Ch. II]{GOL56NILP}),
\begin{theorem} \label{unique writing in products of several factors}
 Let  $\{G_i\}_{i \in \mathcal I}$ be  a family of groups indexed on a totally ordered set $\mathcal I$.  Let $W\subseteq F_\infty$ be a set of words.
 Every element $y \in \2W {i \in \mathcal I} G_i$ admits a unique
 representation $y = a_{i_1}a_{i_2}...a_{i_l}\,u$,
where $a_{i_k}\in G_{i_k}$, $i_1<i_2<...<i_l\in\mathcal I$ , and
$u$ belongs to $[G_i]^{w}$, the projection of $[G_i]^\mathcal F$ onto $\2W {i \in \mathcal I} G_i $.
\end{theorem}

A key ingredient in the proofs of the next section is the notion of the support of an element in the verbal product of arbitrarily many groups. 
While  it is intuitively clear what the support should be, 
to prove that it is well defined requires the next two technical lemmas that can be omitted on a first reading.

\begin{lemma}\label{inclusion in verbal products}
Let $\{G_i\}_{i\in \mathcal{I}}$ be a family of groups indexed in a totally ordered set $\mathcal{I}$. Let $\mathcal{I}_0$ be a subset of  $\mathcal{I}$ and let $W\subseteq F_\infty$ be a set of words. Then the following diagram commutes
\[\begin{tikzcd}
\PRO{\mathcal{I}_0} G_i \arrow{r}{\tilde{i}_{\mathcal{I}_{0}}} \arrow[swap]{d}{q} & \PRO{\mathcal{I}}G_{i} \arrow{d}{\tilde{q}} \\
\2W{\mathcal{I}_{0}} G_i \arrow[dashed]{r}{i_{\mathcal{I}_0}} & \2W{\mathcal{I}}G_{i}
\end{tikzcd}
\]
Moreover, $i_{\mathcal{I}_{0}}$ is injective.
\end{lemma}
\begin{proof}
Note that $i_{\mathcal{I}_0}$ is well-defined since
 \begin{equation}\label{cuenta inclusion de kernels}
 \tilde{i}_{\mathcal{I}_0}(\ker{q} )= \tilde{i}_{\mathcal{I}_0}\Bigg(W\Big(\PRO{{\mathcal{I}_0}} G_i\Big) \cap \big{[}G_i\big{]}^{\COMM\limits_{\mathcal{I}_{0}} G_i}\Bigg)\subseteq W\Big(\PRO{\mathcal{I}} G_i\Big) \cap \big{[}G_i\big{]} ^{\COMM\limits_{\mathcal{I}} G_i}= \ker( \tilde{q}).
\end{equation}
This also shows that  $\tilde{i}_{\mathcal{I}_0}(\ker{q} )\subseteq \ker(\tilde{q}\circ\tilde{i}_{\mathcal{I}_0})$.
Assuming that these sets  are equal, the same argument given at the end of the proof of Proposition \ref{quotient},
 together with the fact that  $ \tilde{i}_{\mathcal{I}_0}$ is injective, shows that $i_{\mathcal{I}_0}$ is injective.
   Let us then prove the reverse inclusion. To that end, let $g\in \ker(\tilde{q}\circ\tilde{i}_{\mathcal{I}_0})$. 
Then,
 \begin{equation*}\tilde{i}_{\mathcal{I}_0}(g)\in W\Big(\PRO{\mathcal{I}} G_i\Big)\bigcap \big{[}G_i\big{]} ^{\COMM\limits_{\mathcal{I}} G_i}\bigcap \tilde{i}_{\mathcal{I}_0}\Big(\PRO{{\mathcal{I}_0}} G_i\Big).\end{equation*}
On the one hand,  by Lemma \ref{single w}, there exists $w \in \mathcal{C}(W)$ a word of length $n$ 
and ${\bf g}\in \Big(\COMM\limits_{\mathcal{I}}G_i\Big)^n$, such that $\tilde{i}_{\mathcal{I}_0}(g)=w(\bf{g})$.
Using the Lemma \ref{ABC} applied to $A:= \PRO{\mathcal{I}_0}G_i$ and $B:= \PRO{ \mathcal{I} \setminus \mathcal{I}_0}G_i$ it follows that 
$w({\bf g}) = \tilde{i}_{\mathcal{I}_0}(w(\pi_A({\bf g})))\tilde{i}_{\mathcal{I} \setminus \mathcal{I}_0}(w(\pi_B({\bf g})))u$, 
with $u\in W\Big(\COMM\limits_{\mathcal{I}}G_i\Big)\cap[A,B]$. 
Moreover, since $ w({\bf g}) \in \tilde{i}_{\mathcal{I}_0}\big(A\big)$, then,
by uniqueness of the writing in Lemma \ref{ABC}, it must be that $w({\bf g}) = \tilde{i}_{\mathcal{I}_0}(w(\pi_A({\bf g})))$ and hence $\tilde{i}_{\mathcal{I} \setminus \mathcal{I}_0}(w(\pi_B({\bf g}))) = 1$ and $u=1$. Then, since $ \tilde{i}_{\mathcal{I}_0}$ is injective, it follows that $g \in W(A )$.

On the other hand, by Theorem \ref{unique writing in products of several factors}, $ \tilde{i}_{\mathcal{I}_0}(g)$ has a writing with elements of
 $\tilde{i}_{\mathcal{I}_0}\Big(\COMM\limits_{\mathcal{I}_0}G_i\Big)$, namely 
 $ \tilde{i}_{\mathcal{I}_0}(g)= a_{i_1} \dots a_{i_l}u$ with $a_{i_k} \in G_{i_k}$, $i_1 < i_2 < \dots< i_l \in \mathcal{I}_0$ and $u \in  \tilde{i}_{\mathcal{I}_0}\Big(\big{[}G_i\big{]}^{\COMM\limits_{\mathcal{I}_0}G_i}\Big)$.
 Moreover, since $ \tilde{i}_{\mathcal{I}_0}(g) \in {\big{[}}G_i\big{]} ^{\COMM\limits_{\mathcal{I}} G_i}$, by uniqueness of the writing in  Theorem \ref{unique writing in products of several factors}, it follows that $a_{i_k} = 1$ for all $1\leq k\leq l$, 
 and  $ \tilde{i}_{\mathcal{I}_0} (g) \in \tilde{i}_{\mathcal{I}_0}\Big(\big{[}G_i\big{]}^{\COMM\limits_{\mathcal{I}_0}G_i}\Big)$.
  Then, since $ \tilde{i}_{\mathcal{I}_0}$ is injective, it follows that $g \in \big{[}G_i\big{]}^{\COMM\limits_{\mathcal{I}_0}G_i}$.
\end{proof}

\begin{remark}\label{write with finite support}Given $y\in \2W {i \in \mathcal I} G_i$, there exists  $g\in \PRO{i \in \mathcal I} G_i$ such that $\tilde{q}(g)=y$.
Then, by definition of the free product of groups, 
there exists a finite subset $\mathcal I_{0}$ of $\mathcal I$, such that 
$g\in \tilde{i}_{\mathcal{I}_0}\big(\COMM_{i \in \mathcal I_{0}} G_i\big)$
Hence, by the previous lemma, we can think of $y$ as being an element of $\2W {i \in \mathcal I_{0}} G_i$.
\end{remark}
\begin{lemma}\label{intersection of verbal products}
Let $\{G_i\}_{i \in \mathcal{I}}$ be a family of groups indexed in a totally ordered set $\mathcal{I}$. 
 Let $\mathcal{I}_1$ and $ \mathcal{I}_2$ be subsets of  $\mathcal{I}$ and
  let $W \subseteq F_\infty$ be a set of words.
   Consider $\2W{\mathcal{I}_1} G_i$ and $\2W{\mathcal{I}_2} G_i$ viewed as subgroups of $\2W{\mathcal{I}} G_i$, 
   according to Remark \ref{write with finite support}.
  Then
$\Big(\2W{\mathcal{I}_1} G_i \Big)\bigcap \Big( \2W{\mathcal{I}_2} G_i \Big)= \2W{\mathcal{I}_1\cap \mathcal{I}_2} G_i$,
 where we understand the verbal product over the empty set as  the element $1$.
\end{lemma}
\begin{proof}
It is clear that $\2W{\mathcal{I}_1\cap \mathcal{I}_2} G_i \subseteq \Big(\2W{\mathcal{I}_1} G_i \Big)\cap \Big( \2W{\mathcal{I}_2} G_i \Big)$. 
For the reverse inclusion, take $y \in \Big(\2W{\mathcal{I}_1} G_i \Big)\cap \Big( \2W{\mathcal{I}_2} G_i \Big)$.
Then, as $y \in \2W{\mathcal{I}_1} G_i$, by Theorem \ref{unique writing in products of several factors}, $y$ 
has a unique writing by elements of $\mathcal{I}_1$ and as
 $y \in \2W{\mathcal{I}_2} G_i$, it has a unique  writing by elements of $\mathcal{I}_2$. 
These two expressions must be the same since $y$ can be viewed as an element of $\2W{\mathcal{I}_{1}\cup \mathcal{I}_{2}} G_i$.
 Then $y \in \2W{\mathcal{I}_1\cap \mathcal{I}_2} G_i$.
\end{proof}

Remark \ref{write with finite support} says that for any $y\in \2W {i \in \mathcal I} G_i$,
 there exists a finite subset  $\mathcal I_{1}$ such that $y\in \2W {i \in \mathcal I_{1}} G_i$. 
 Consider the index set $\mathcal I_{0}:=\bigcap \big\{\tilde{\mathcal I}\subseteq \mathcal I_{1}: y\in \2W {i \in \tilde{\mathcal I}}G_i\big\}$.
 Since $\mathcal I_{1}$ is finite,   Lemma \ref{intersection of verbal products} makes clear that $y\in \2W {i \in \mathcal I_{0}} G_i$.
If there exists $ \mathcal I_{2}\subseteq  \mathcal I$ such that $y\in \2W {i \in \mathcal I_{2}}G_i$, then
by Lemma \ref{intersection of verbal products}, $y\in \2W {i \in \mathcal I_{0}\cap\mathcal I_{2}}G_i$. 
By the minimality of $\mathcal I_{0}$ as a subset of the finite set $\mathcal I_{1}$, it follows that 
$\mathcal I_{0}\cap\mathcal I_{2}=\mathcal I_{0}$.
This allows  to define the support of an element in the verbal product of groups.
\begin{definition}  \label{defsop}
Given $y \in \2W{i \in \mathcal I} G_i$, its support is the smallest subset $\mathcal I_{0}$ of $\mathcal I$ such that
$y\in\2W{i \in \mathcal I_{0}} G_i$.
\end{definition}
\begin{remark}\label{gauge} Some obvious properties of the support are:
\begin{enumerate}[label=(\arabic*)]
\item $supp(y)$ is a finite set;
\item $supp(y)=supp(y^{-1})$;
\item $supp (yy^\prime)\subseteq supp (y)\cup supp(y^\prime)$;
\item $supp(y)$ is empty if and only if $y=1$;
\end{enumerate}
\end{remark}

A fundamental property of verbal products is that it is an associative operation on groups.  For a proof of this fact we refer to \cite[Theorem 5.1]{GOL56NILP} for the case of nilpotent products and to \cite[Section $\S$ 5]{MOR56} for arbitrary verbal products. 
The associativity of the verbal product allows us to prove Corollary \ref{amenability of many products} from in the introduction. 

\begin{proof}[Proof of Corollary \ref{amenability of many products}] If $\mathcal I$ is finite, the result follows from associativity together with Proposition \ref{StabProp}. If $\mathcal I=\mathbb N$, then 
\begin{equation*}\2W {i\in\mathbb N } G_i=\bigcup_{n\in\mathbb N} \2W {i\in\{1,2,\ldots,n\} } G_i\end{equation*}
and amenability, soficity, hyperlinearity, weak soficity, linear soficity, the Haagerup property and exactness are preserved under countable increasing unions of discrete groups (see \cite[Proposition G.2.2]{BEK08}, \cite[Theorem 1]{ELE03},\cite[Proposition 6.1.1]{CHER01} and \cite[Exercise 5.1.1]{BRO08}). 
\end{proof}
\begin{remark} Property (T) is not stable under taking  verbal products of infinitely many discrete groups with more than one element.
 This is because such a group is not finitely generated.
\end{remark}

\section{Restricted verbal wreath products of groups} \label{section 5}

We start this section by recalling the definition of the restricted verbal wreath product between two groups. As it was already explained in the introduction, this notion was first introduced by  Shmelkin in \cite{MR0193131}. 

Let $G$ and $H$ be countable groups. Let $W\subseteq F_{\infty}$ be a set of words. Let $\mathcal F:=\COMM_{H} G$ be the free product of $|H|$-many copies of $G$. There is an action $H \stackrel{\alpha}{\curvearrowright} \mathcal F$, given by permuting the copies of $G$, that is, if $(g)_{h_1}$ denotes the element $g$ in the copy $h_1$ of $G$ in $\mathcal F$, then $\alpha_h((g)_{h_1})= (g)_{hh_1}$. 

Due to Lemma \ref{fullyinvarianceforwords} and to the fact that  for each $h\in H$, $\alpha_h$ is an automorphism, 
we have that the set $W(\mathcal F)\cap [G_i]^\mathcal F$ is invariant under the action of $H$.
Hence, there is a well-defined action $H \stackrel{\alpha}{\curvearrowright} \2W{H} G$.
\begin{definition} Let $G$ and $H$ be countable groups. Let $W\subseteq F_{\infty}$ be a set of words.
The semi-direct product  
$ G\wrw H:=\Big (\2W {H}G\Big )\rtimes_\alpha H$  will be called {\it the restricted verbal wreath product} of $G$ and $H$.
\end{definition}

\begin{remark} \label{cuatro} It follows from Definition \ref{defsop} that 
\begin{equation*}supp(\alpha_{h}(y))=h supp(y),\text{ for all }h\in H
\text { and for all }y\in \2W{H} G.\end{equation*}
\end{remark}

In \cite{SAS18}, it was shown that for $W=\{n_2\}$  as in example \ref{examples of products}\ref {item2examples product}, the restricted verbal wreath product of two groups with the Haagerup property has the Haagerup property. This was done following the general strategy developed by Cornulier, Stalder and Valette in \cite{CSV2}. The only properties of the second nilpotent product used in the proof were the associativity, the notion of support and the fact that the second nilpotent product between two groups with the Haagerup property has the Haagerup property. Exactly the same proof presented in \cite[Section $\S$5]{SAS18} allows us to prove Theorem \ref{theo2}.

\begin{proof}[Sketch of proof of Theorem \ref{theo2}] The five properties of Remarks \ref{gauge} and \ref{cuatro} are enough to prove adequate variants of \cite[Example 5.5 and Proposition 5.7]{SAS18}.
Then, combine them with \cite[Theorem 5.1]{CSV2}.
\end{proof}
\begin{remark} Since amenable-by-amenable extensions are amenable, it follows that in the case that $G$ and $H$ are both amenable, then  $G\wrw H$ is amenable for nilpotent, solvable, and $k$-Burnside wreath products, for $k=2,3,4,6$.
\end{remark}
\vskip .2in

\subsection{Sofic verbal wreath products}

As it was mentioned in the introduction, in  \cite{HAY18}, Hayes and Sale showed that the 
restricted wreath product of sofic groups is a sofic group.
In this subsection we will adapt the strategy developed in \cite{HAY18} 
to prove that certain restricted verbal wreath products of sofic groups are sofic. Before giving further explanations, 
let us recall the definition of sofic groups.

\begin{definition}
A countable discrete group $G$ is said to be sofic if for every $\varepsilon>0$ and every $F \subseteq G$ finite set, there exist a finite set $A$ and a function $\phi: G \to Sym(A)$ satisfying that $\phi(1)=1$ and 
\begin{itemize}
\item $(F,\varepsilon,d_{\rm{Hamm}})$-multiplicative: for all $g,g'\in F$  we have that $d_{\rm{Hamm}}(\phi(g)\phi(g'),\phi(gg')) < \varepsilon$;
 
 \item $(F,\varepsilon,d_{\rm{Hamm}})$-free: if $g\in F\setminus \{1\}$ we have that $d_{\rm{Hamm}}(\phi(g),1) \geq  1-\varepsilon$;
\end {itemize} 
 where the normalized Hamming distance in $Sym(A)$ is given by
\begin{equation*}d_{\rm{Hamm}}(\sigma,\tau):= \frac{1}{|A|}|\{a\in A: \sigma(a)\neq \tau(a)\}|\end{equation*}
\end{definition}

Starting with sofic approximations of $G$ and $H$ respectively,  Hayes and Sale provided an explicit sofic approximation of $G\wr H$.
To that end, built in their proof, there are explicit sofic approximations of finite direct sums of the form $\bigoplus_B G$ constructed from sofic approximations 
$\phi:G \to Sym(A)$. 
This is easily done by defining
\begin{align*}
\Theta: \bigoplus_B G &\to \bigoplus_B Sym(A)\stackrel{diag}\hookrightarrow Sym(A^{|B|})\\
(g_b)_{b\in B} &\mapsto (\phi(g_b))_{b\in B}
\end{align*}
However, if we try to replicate this in the case of verbal products,  additional technical difficulties arise.
The problem being that there is no obvious way on how to define sofic approximations on elements of $\big[G\big]^{\COMM\limits^{_w}_{B}G}$.
Indeed, a first attempt would be to start with a sofic approximation of $G$ on $Sym(A)$ 
and construct a {\it coordinate-wise} sofic approximation
$\Theta$, of $\COMM\limits_{B}^{_w}G $ on $\bigoplus_B Sym(A)$.
The obstruction is that $\Theta([g_i,g_j])=1$, for every $g_i\in G_i$, $g_j\in G_j$ with $i\neq j$, while $[g_i,g_j]$ 
is in general nonzero for $k$-nilpotent products when $k\geq 2$.

For a second attempt, we could consider approximating by the group $Sym(A) \vW Sym(A)$ instead of by the group $Sym(A) \times Sym (A)$. The following example shows this does not work either.
Let $p$ be an odd prime, consider the homomorphism
 \begin{align*}
\theta: \nicefrac{\mathbb{Z}}{p\mathbb{Z}} &\to Sym(A) \text{, where }A=\{1,\dots,p\}\\
1&\mapsto (1\dots p)\text{, the cycle of length } p
\end{align*} 
let  $\Theta_0: \nicefrac{\mathbb{Z}}{p\mathbb{Z}} \ast\nicefrac{\mathbb{Z}}{p\mathbb{Z}}
\to 
Sym(A) \ast Sym(A)$ be the free product homomorphism and let 
$\Theta: \nicefrac{\mathbb{Z}}{p\mathbb{Z}} \stackrel{_2}{\ast}\nicefrac{\mathbb{Z}}{p\mathbb{Z}}
\to 
Sym(A) \stackrel{_2}{\ast} Sym(A)$ be the quotient homomorphism. 
From \cite[Proposition 2.13]{SAS18}, we have 
\begin{equation*}[\nicefrac{\mathbb{Z}}{p\mathbb{Z}} ,\nicefrac{\mathbb{Z}}{p\mathbb{Z}}]^
{\nicefrac{\mathbb{Z}}{p\mathbb{Z}} \,\stackrel{_2}{\ast}\,\nicefrac{\mathbb{Z}}{p\mathbb{Z}}}\cong \nicefrac{\mathbb{Z}}{p\mathbb{Z}}
\,\,\,\,\,\,\,\,\,\,\text{ and }\,\,\,\,\,\,\,\,\,\,
[Sym(A) , Sym(A)]^{Sym(A)\stackrel{_2}{\ast} Sym(A)}\cong \nicefrac{\mathbb{Z}}{2\mathbb{Z}}\end{equation*}
It follows that if we take $F:=[\nicefrac{\mathbb{Z}}{p\mathbb{Z}} ,\nicefrac{\mathbb{Z}}{p\mathbb{Z}}]^
{\nicefrac{\mathbb{Z}}{p\mathbb{Z}} \,\stackrel{_2}{\ast}\,\nicefrac{\mathbb{Z}}{p\mathbb{Z}}}$ then $F\subseteq \ker(\Theta)$
 and hence, $\Theta$ cannot be $(F,\varepsilon)$-free for any $\varepsilon>0$. 

The examples above hint that it is difficult to spell an explicit sofic approximation of the verbal product $\2W{B} G$ starting from a sofic approximation of $G$.
However, by means of Corollary \ref{amenability of many products}, we know that for every $F\subseteq \2W {B}G$ finite set and $\varepsilon>0$, there exists a $(F,\varepsilon)$-sofic approximation of $\2W{B} G$. 
Knowing the existence of a sofic approximation of the verbal product without passing to its construction, simplifies some technical steps of the proof of Theorem  \ref{theo sofic intro}. In the case of the verbal product for the word $\{n_1\}$ of example  \ref{examples of products}\ref{item1examples product}, this simplifies a bit the proof in \cite{HAY18}. Of course, on the negative side, our proof gives, in principle, less information than \cite{HAY18}. Another difference with respect to \cite{HAY18} is that in our proof we take the point of view discussed in \cite[Remark 3.5]{HAY18} and thus we deal with what Holt and Rees called {\it strong discrete $\mathcal {C}$-approximations} \cite[$\S$ 1]{MR3632893}. We believe this approach further simplifies the proof.
 Before starting with the proof of Theorem  \ref{theo sofic intro}, 
we recall the following Lemma from \cite{HAY18},  adapted to the situation at hand. 

\begin{lemma}\cite[Lemma 2.8]{HAY18} \label{cuasimultiplicativa}
Let $G$ and $H$ be countable, discrete groups, and let 
 \begin{align*}
  {\rm proj}_H: G\wrw H &\to H,\\
{\rm proj}_G: G\wrw H &\to \2W{H} G ,   
 \end{align*}
be the canonical projection maps. For a finite subset $F_0 \subseteq G\wrw H$  with $1\in F_0$ define the subsets
\begin{align}
E_1& := \{\alpha_h(x): h \in {\rm proj}_H(F_0),x \in {\rm proj}_G(F_0) \} \label{def of E1};\\
\tilde{E}_1& := \{y\alpha_h(x): h \in {\rm proj}_H(F_0), x,y \in {\rm proj}_G(F_0) \}\label{def of E1 tilde};\\
E_2 & :=   {\rm proj}_H(F_0).\label{def of E2}
\end{align}
Let $\varepsilon >0$ and let $(K,d)$ be a group with bi-invariant metric $d$. 
Suppose $\Gamma: G \wrw H \to K$ is a function with $\Gamma(1)=1$  that verifies the following properties:
\begin{enumerate}[label=(\roman*)]
    \item \label{item1} the restriction of $\Gamma$ to $\2W{H}G$ is $(E_1, \nicefrac{\varepsilon}{6},d)$-multiplicative;
    \item \label{item2} the restriction of $\Gamma$ to $H$ is $(E_2, \nicefrac{\varepsilon}{6},d)$-multiplicative;
      \item \label{item3} $\max_{x\in \tilde{E}_{1}, h \in E_{2}.E_2}  d(\Gamma(x,1)\Gamma(1,h),\Gamma(x,h))< \varepsilon/6;$
       \item \label{item4}$\max_{x\in {\rm proj}_G(F_0), h \in E_2}   d(\Gamma(1,h)\Gamma(x,1),\Gamma(\alpha_h(x),1)\Gamma(1,h)) < \varepsilon/6$.
\end{enumerate}
\noindent
Then $\Gamma$ is $(F_0,\varepsilon,d)$-multiplicative, namely for all 
$z,z^\prime\in F_0,   d(\Gamma(z)\Gamma(z^\prime);\Gamma(zz^\prime))<\varepsilon$.
\end{lemma}

\begin{proof}
For $(x,h),(x',h')\in F_0$,   the triangular inequality, the invariance of $d$ and properties \ref{item1},\ref{item2},\ref{item3} yield the following estimates
\begin{align*}
 &d(\Gamma(x,h)\Gamma(x',h'); \Gamma(x,1)\Gamma(1,h)\Gamma(x',1)\Gamma(1,h'))<\varepsilon/3;\\
&d(\Gamma(x\alpha_h(x'),hh');\Gamma(x,1)\Gamma(\alpha_h(x'),1)\Gamma(1,h)\Gamma(1,h') )<\varepsilon/2.
\end{align*}
Then,
\begin{align*}
    d(\Gamma(x,h)& \Gamma(x',h');\Gamma(x\alpha_h(x'),hh')) \\ 
&\leq d(\Gamma(x,1)\Gamma(1,h)\Gamma(x',1)\Gamma(1,h');\Gamma(x\alpha_h(x'),hh')) + \varepsilon/3 \\
&\leq d(\Gamma(x,1)\Gamma(1,h)\Gamma(x',1)\Gamma(1,h');\Gamma(x,1)\Gamma(\alpha_h(x'),1)\Gamma(1,h)\Gamma(1,h')) + \frac{5}{6} \varepsilon \\
 &= d(\Gamma(1,h)\Gamma(x',1); \Gamma(\alpha_h(x'),1)\Gamma(1,h)) + \frac{5}{6}\varepsilon< \varepsilon\,\,\,\,\,\,\text{(here we use \ref{item4})}.\qedhere
\end{align*}
 \end{proof} 
 
 \begin{remark}
  Item \ref{item3} here is slightly different from 
\cite[Lemma 2.8, third bullet]{HAY18}. We do not know how to prove the lemma only using the hypothesis of \cite{HAY18} on item \ref{item3}.
This does not affect the proof of  Theorem \ref{theo sofic intro}, since, like in \cite{HAY18}, when we apply Lemma \ref{cuasimultiplicativa}, we will prove that in the cases at hand it holds that
$\Gamma(x,1)\Gamma(1,h)=\Gamma(x,h)$ for all $x\in \2W{H} G $ and all $h\in H$.
\end{remark}

\begin{proof}[Proof of Theorem \ref{theo sofic intro}]
Let $F \subseteq G\wrw H$ be a finite subset and $\varepsilon >0$. Define $F_0:= F \cup \{1\}\cup F^{-1}$ and consider the sets $E_1$ and $E_2$ as in Lemma \ref{cuasimultiplicativa}. Also, consider the sets
\begin{align}
E&:= E_2\cup E_2\cdot supp(E_1)  \label{def of E}; \\
E_H&:= E \cdot E^{-1} \label{def of E sub H};
\end{align} 
where, as usual, $ supp(E_1)=\bigcup_{x\in E_{1}} supp(x)$. 

Since $H$ is sofic,  for any $\varepsilon'>0$, 
there exist  a finite set $B$ and a 
$(E_H,\varepsilon')\text{-sofic approximation }$
\begin{equation*}\sigma: H \to Sym(B)\text{  with }\sigma(1)=1.\end{equation*} 
Define the sets 
\begin{align}
    B_1&:=\{b\in B: \sigma(h_1)^{-1}b \neq \sigma(h_2)^{-1}b \;\text{for all } h_1 \neq h_2\in E\};   \label{def of  B 1}\\
B_2 &:= \{ b\in B : \sigma(h_2 h_1)^{-1}b = \sigma(h_1)^{-1}\sigma(h_2)^{-1}b,\; \text{for all } h_1,h_2 \in E\};  \label{def of  B 2}\\
 B_E&:= B_1 \cap B_2; \label{def of B sub E}
\end{align}
and  recall the following lemma from \cite{HAY18}.
\begin{lemma}\cite[Lemma 3.4]{HAY18}
\label{kappa}
Let $\kappa>0$. If $\varepsilon'<\frac{\kappa}{4|E|^2}$, then
\begin{equation*}\frac{|B\setminus B_E|}{|B|}\leq \kappa.\end{equation*}
\end{lemma}
 \begin{remark}\label {choice of epsilon prime}
 In what follows, $\varepsilon'$ will be chosen to be dependent on $\kappa$ according to Lemma \ref{kappa}, 
 namely we will choose $\varepsilon'<\frac{\kappa}{4|E|^2}$. 
\end{remark}

For each $h\in H$ and each $b\in B$, call $\theta_b^{(h)}: G \to \COMM\limits_{B}G$,  the embedding of $G$ into the $\sigma(h)^{-1}b$ copy of $G$ inside the free product
  $\COMM\limits_{B}G$,  and define
   \begin{align}
        \tilde{\theta}_b: \COMM\limits_{E} G & \to \COMM\limits_{B}G \nonumber\\
        (g_{h_1})_{h_1}(g_{h_2})_{h_2}\ldots (g_{h_n})_{h_n}
 & \mapsto \theta_b^{(h_1)}(g_{h_1})\dots \theta_b^{(h_n)}(g_{h_n})\label{theta b in free product},
   \end{align}
where $(g_{h_i})_{h_i}$ denotes the element  $g_{h_i}$ in the   $h_i$  copy of $G$ inside  $\COMM\limits_{E} G$, and $h_i\neq h_{i+1}$.\\
 By the universal property of the free product, $ \tilde{\theta}_b$ is a group homomorphism.
  From it, we will construct a homomorphism between the groups $\COMM\limits_{E}^{_w} G$ 
  and $\COMM\limits^{_w}_{B}G$. To that end, let   $w\in W\subseteq F_\infty$ be a word in $n$ 
letters and let $(y_1,\dots,y_n)\in \Big( \COMM\limits_{E} G \Big)^n$.
 By Lemma \ref{fullyinvarianceforwords}, we have that 
 $\tilde{\theta}_b(w(y_1,\dots,y_n)) = w(\tilde{\theta}_b(y_1),\dots, \tilde{\theta}_b(y_n))\in W\Big( \COMM\limits_{B} G \Big).$
 Also, if $g_{h_{i}}\in G_{h_i}$ and $g_{h_{j}}\in G_{h_j}$, then 
 $\tilde{\theta}_b([g_{h_{i}},g_{h_{j}}])=[\tilde{\theta}_b(g_{h_{i}}),\tilde{\theta}_b(g_{h_{j}})]\in [G_{\sigma^{-1}(h_i)b},G_{\sigma^{-1}(h_j)b}].$
 Finally, in the case when $b\in B_1$ and  $h_i \neq h_j$,
we have that $[G_{\sigma^{-1}(h_i)b},G_{\sigma^{-1}(h_j)b}]$ is inside  of  $\big[G\big]^{\COMM\limits_{B} G}$ 
(here we used the notation in Definition \ref{mixed commutator}).
Hence, all this combined tells that if $b\in B_1$ and
$ u \in W\Big(\COMM\limits_{E }G\Big) \cap 
\big[G\big]^{\COMM\limits_{E} G}$,
then 
$\tilde{\theta}_b (u)  \in W\Big(\COMM\limits_{B}G\Big) \cap \big[G\big]^{\COMM\limits_{B} G}.$ 
This shows that for each $b \in B_1$, we have a well-defined quotient homomorphism
  \begin{equation}\label{theta b is a homomorphism in E}
  \theta_b: \2W{E} G \to \2W{B}G.
  \end{equation}
 
  We claim that for $b\in B_1$, $\theta_b$ is injective. Indeed,  
  define  $T_b:= \{\sigma(h)^{-1}b : h \in E\}\subseteq B$. Since $b\in B_1$, the sets $E$ and $T_b$ have the same (finite) cardinal. 
 Then $\theta_b$ can be regarded as the composition of the following injective homomorphisms
  \begin{equation*} \2W{E} G \,\,\cong \,\, \2W{T_{b}}G\,\,\hookrightarrow \,\,\2W{ B}G.\end{equation*}
  
By Remark \ref{write with finite support}, we can regard $\2W{E} G$ as a subgroup of $\2W{H} G$.
  We now extend  $\theta_b$ to a function  (not a homomorphism) 
  \begin{equation*}\theta_b: \2W{H} G \to \2W{B}G\end{equation*}
  by declaring 
  \begin{align}\label{def-of-theta-b}
        \theta_b(x):=
  \begin{cases}
\theta_b(x)  &\text{ if  } x\in  \2W{E} G  \\
1 \hspace{0.65cm}  &\text{ if not }
\end{cases}
  \end{align}
and define
\begin{align*}
    \theta_B : \2W{H} G &\to \bigoplus\limits_{B} \2W{B}G\\
 x &\mapsto \Bigg( b\mapsto 
\begin{cases}
\theta_b(x)  &\text{ if  } b\in B_E=B_1\cap B_2  \\
1 &\text{ if not }
\end{cases}\Bigg)
\end{align*}
\begin{remark} The condition $b\in B_E$ rather than $b\in B_1$ will be necessary only in  \eqref{why we need b2}.
\end{remark}
 $Sym(B)$ acts by permutations on the summands of $\bigoplus\limits_{B} \Big(\2W{B}G\Big)$. 
The permutational wreath product $\Big(\bigoplus\limits_{B} \2W{B}G\Big)\rtimes Sym(B)$ is denoted by $\Big(\2W{B}G\Big) \wr_B Sym(B)$.
Finally, define
\begin{align*}
        \Theta: G\wrw H  & \to \Big(\2W{B}G\Big) \wr_B Sym(B) \\
    (x,h) &\mapsto (\theta_B(x), \sigma(h))
\end{align*}

Since $G$ is sofic, the hypothesis in the statement of  Theorem \ref{theo sofic intro}, together with Corollary \ref{amenability of many products}, implies that the verbal product $\2W{B }G$ is a sofic group. 
Hence, given the finite set
\begin{equation}\label{E sub g}E_G:= \bigcup_{b\in B_{E}} \theta_b(E_1)
\end{equation}
and $\varepsilon'>0$,
  there exist a finite set $A$ and $\varphi: \2W{B}G \to Sym(A)$ a $(E_G,\varepsilon')$-sofic approximation of $\2W{B}G$ with $\varphi(1)=1$. Define 
\begin{align*}
    \varphi_\star:  \2W{B} G \wr_B Sym(B) &\to Sym(A)\wr_B Sym(B) \\
    ((x_b)_{b\in B},\tau) &\mapsto((\varphi(x_b))_{b\in B},\tau)
\end{align*}
and consider the embedding
\begin{align*}
    \psi:Sym(A)\wr_B Sym(B) &\to Sym(A \times B) \\
\psi(\alpha,\beta)(a,b) &=((\alpha)_{\beta(b)}(a),\beta(b)).
\end{align*}
Finally, define the function
   $ \Gamma : G \wrw H \to Sym(A \times B)
\text{ given by } \Gamma(x,h):=\psi(\varphi_\star(\Theta(x,h))$.
To be more explicit, 
\begin{align}\label{gamma explicit}
    \Gamma(x,h)(a,b)=
  \begin{cases}
((\varphi \theta_{\sigma(h)b}(x))(a), \sigma(h)b)  &\text{ if  } \sigma(h)b\in B_E \\
(a,\sigma(h)b)  &\text{ if }\sigma(h)b\notin B_E
\end{cases}
\end{align}

{\bf Claim:} $\Gamma$ is a $(F_0,\varepsilon)$-sofic approximation of $G \wrw H$.\\
 In order to prove that $\Gamma$ is $(F_0,\varepsilon,d_{\rm{Hamm}})$-multiplicative, it is enough to show that the four premises of Lemma \ref{cuasimultiplicativa} hold true. 
To that end, consider $K=Sym(A \times B)$ with its Hamming distance and observe that $\Gamma(1,1)=1$.

In order to check  \ref{cuasimultiplicativa}\ref{item1}, take $(x,1), (x',1)$ with $x,x' \in E_1$. Note that 
\begin{equation*}
supp(x), supp(x'), supp(xx')\subseteq supp(E_1)\subseteq E. \text{ In particular } x,x',xx'\in \2W{E}G.
\end{equation*} 
On the one hand, if $b\notin B_E$,  by \eqref{gamma explicit}, we have that $\Gamma(x,1)(a,b)=(a,b)$, for every $x\in\2W{H}G$.
On the other hand, if $(a,b) \in A\times B_E$, by \eqref{gamma explicit}, we have the identities
   \begin{equation*} (\Gamma(x,1)\Gamma(x',1))(a,b)=\Gamma(x,1)\Big(\varphi(\theta_b(x'))(a), b\Big)= 
\Big((\varphi(\theta_b(x))\circ \varphi(\theta_b(x')))(a), b\Big)\end{equation*}
and
\begin{equation*}
\Gamma(xx',1)(a,b)=\big( \varphi(\theta_b(xx'))(a),b \big).\end{equation*}
Then 
\begin{align*}
    d_{\text{Hamm}}&(\Gamma(x,1)\Gamma(x',1),\Gamma(xx',1))\\
 &=\frac{1}{|A||B|} |\{(a,b) \in A\times B: \Gamma(x,1)(\Gamma(x',1)(a,b)) \neq \Gamma(xx',1)(a,b)\}|\\
 &=\frac{1}{|A||B|} |\{(a,b) \in A\times B_E: \Gamma(x,1)(\Gamma(x',1)(a,b)) \neq \Gamma(xx',1)(a,b)\}|\\
 &=\frac{1}{|A||B|} |\{(a,b) \in A\times B_E: \Big(\varphi(\theta_b(x))\big(\varphi(\theta_b(x'))a\big), b\Big) \neq (\varphi(\theta_b(xx'))a,b)\}|\\
&=\frac{1}{|B|}\sum\limits_{b\in B_E} \frac{1}{|A|}|\{a\in A: \varphi(\theta_b(x))\varphi(\theta_b(x'))a\neq \varphi(\theta_b(xx'))a\}|\\
&=\frac{1}{|B|}\sum_{b\in B_E}d_{\text{Hamm}}(\varphi(\theta_b(x))\varphi(\theta_b(x')),\varphi(\theta_b(xx')))\\
&=\frac{1}{|B|}\sum_{b\in B_E}d_{\text{Hamm}}(\varphi(\theta_b(x)\varphi(\theta_b(x'))),\varphi(\theta_b(x)\theta_b(x'))),
\end{align*}
where the last equality is valid because, by \eqref{theta b is a homomorphism in E},   $\theta_b$ is a group homomorphism in $\COMM\limits_{E}^{_w}G$ and  $x,x'\in \COMM\limits_{E}^{_w}G$.
Since $\varphi$ is $(E_G,\varepsilon',d_{\text{Hamm}})$-multiplicative and since, by \eqref{E sub g}, $\theta_b(x),\theta_b(x')\in E_G$,
it follows that
\begin{equation*}d_{\text{Hamm}}(\Gamma(x,1)\Gamma(x',1),\Gamma(xx',1))<\varepsilon'\frac{|B_E|}{|B|}<\varepsilon'<\kappa<\varepsilon/6,\end{equation*}
once we choose $\kappa<\varepsilon/6$ in Lemma \ref{kappa} and use Remark \ref{choice of epsilon prime}.\\

In order to check  \ref{cuasimultiplicativa}\ref{item2}, recall that for any $h\in H$,  \eqref{gamma explicit} implies that \begin{equation*}\Gamma(1,h)(a,b)=(a,\sigma(h)b).\end{equation*} 
Then
\begin{equation*}d_{\text{Hamm}}(\Gamma(1,h)\Gamma(1,h'),\Gamma(1,hh')) = d_{\text{Hamm}}(\sigma(h)\sigma(h'),\sigma(hh')).\end{equation*}
 Since $\sigma$ is $(E_H,\varepsilon',d_{\text{Hamm}})$-multiplicative, and since by \eqref{def of E sub H} we have that $E_2\subseteq E\subseteq E_H$, then for $h,h' \in E_2$,
\begin{equation*}d_{\text{Hamm}}(\Gamma(1,h)\Gamma(1,h'),\Gamma(1,hh'))  < \varepsilon'<
\kappa<\varepsilon/6,\end{equation*}
once we choose $\kappa<\varepsilon/6$ in Lemma \ref{kappa} and use Remark \ref{choice of epsilon prime}.\\
 In order to check  \ref{cuasimultiplicativa}\ref{item3}, a straightforward  computation shows that for any $(x,h)\in G \wrw H$
\begin{equation*}\Gamma(x,h)(a,b)=(\Gamma(x,1) \Gamma(1,h))(a,b),\text { so }
\Gamma(x,h)=\Gamma(x,1) \Gamma(1,h).\end{equation*}

In order to check  \ref{cuasimultiplicativa}\ref{item4},
observe  that for any $(x,h)\in   G \wrw H$, equation \eqref{gamma explicit} gives the following identity 
\begin{align}\label{formula de gamma en b}
(\Gamma(1,h) \Gamma(x,1))(a,b)
 = 
\begin{cases}
(\varphi\theta_b(x)(a),\sigma(h)b) &\text{ if } b\in B_E;\\
(a,\sigma(h)b)   &\text { if } b\notin B_E.\\
\end{cases}
\end{align}
Now let $x\in {\rm proj}_G(F_0), h \in E_2$. Since $supp (x)\subseteq supp(E_1)$,  $x$ is of the form 
\begin{equation*}x=(g_{h_1})_{h_1}(g_{h_2})_{h_2}\ldots (g_{h_n})_{h_n},\text{ with }h_i\in supp(E_1),\end{equation*}
 where $(g_{h_i})_{h_i}$ denotes the element  $g_{h_i}$
  in the   $h_i$  copy of $G$ inside 
  $\2W{supp(E_1)}G\hookrightarrow\2W{E} G\hookrightarrow \2W{H} G$.
Then
\begin{equation*}\alpha_h(x)=\alpha_h\big((g_{h_1})_{h_1}(g_{h_2})_{h_2}\ldots (g_{h_n})_{h_n}\big)=(g_{h_1})_{hh_1}(g_{h_2})_{hh_2}\ldots (g_{h_n})_{hh_n}.\end{equation*}
Observe that since $h\in E_2$, then, by \eqref{def of E}, we have that $hh_i\in E$, so $\alpha_h(x)\in \2W{ E} G\hookrightarrow \2W{H} G.$
Hence, if $\sigma(h)b\in B_E$, then by \eqref{theta b in free product}, \eqref{theta b is a homomorphism in E} and \eqref{def-of-theta-b}, we have that
\begin{align*}
\theta_{\sigma(h)b}(\alpha_h(x)) &=\theta_{\sigma(h)b}((g_{h_1})_{hh_1}(g_{h_2})_{hh_2}\ldots (g_{h_n})_{hh_n})\\
&=\theta_{\sigma(h)b}^{(hh_1)}(g_{h_1}) \theta_{\sigma(h)b}^{(hh_2)}(g_{h_2})\ldots \theta_{\sigma(h)b}^{(hh_n)}(g_{h_n})\\
&=(g_{h_1})_{\sigma(hh_1)^{-1}\sigma(h)b}(g_{h_2})_{\sigma(hh_2)^{-1}\sigma(h)b}\ldots (g_{h_n})_{\sigma(hh_n)^{-1}\sigma(h)b}\\
&=(g_{h_1})_{\sigma(h_1)^{-1}b}(g_{h_2})_{\sigma(h_2)^{-1}b}\ldots (g_{h_n})_{\sigma(h_n)^{-1}b}\, ,
\end{align*}
where the last equality  is valid because  $h_i,h\in E$ 
and $\sigma(h)b\in B_E\subseteq B_2$, then by \eqref{def of B 2}, we have that
\begin{equation} \label{why we need b2}
\sigma(hh_i)^{-1}\sigma(h)b=\sigma(h_i)^{-1}\sigma(h)^{-1}\sigma(h)b=\sigma(h_i)^{-1}b.
\end{equation}
Moreover, if $b\in B_E$ then  by \eqref{def-of-theta-b}, we have that
 \begin{equation*}\theta_{b}(x)=(g_{h_1})_{\sigma(h_1)^{-1}b}(g_{h_2})_{\sigma(h_2)^{-1}b}\ldots (g_{h_n})_{\sigma(h_n)^{-1}b}.\end{equation*}
Hence,  if $b\in B_E$ and $\sigma(h)b\in B_E$ then 
\begin{equation}\label{igualdad soficidad punto 4}
\theta_{\sigma(h)b}(\alpha_h(x))=\theta_{b}(x).
\end{equation}
This entails that if $b\in B_E\cap\sigma(h)^{-1}B_E$ then
\begin{align*}
 \Gamma(\alpha_h(x),1)(\Gamma(1,h)(a,b))&= \Gamma(\alpha_h (x),1)(a,\sigma(h)b)\\
 &=(\varphi\theta_{\sigma(h)b}(\alpha_{h}(x))(a),\sigma(h)b)\\
 &=(\varphi\theta_b(x)(a),\sigma(h)b)\\
 &=\Gamma(1,h)\Gamma(x,1))(a,b)  \,\,\,\,\,\,\,\,\text{ {\it(here we use \eqref{formula de gamma en b})}} .
 \end{align*}
With all this at hand, we now proceed to estimate the Hamming distance in  \ref{cuasimultiplicativa}\ref{item4}.
\begin{align*}
  &d_{\text{Hamm}}(\Gamma(1,h)\Gamma(x,1), \Gamma(\alpha_h(x),1)\Gamma(1,h))\\
&=\sum_{b\in B}\frac{1}{|A||B|}|\{a\in A: \Gamma(1,h)\Gamma(x,1)(a,b)\neq \Gamma(\alpha_h(x),1)\Gamma(1,h)(a,b)\}|\\ 
&=\frac{1}{|A||B|}\sum_{b\in (B_E\cap\sigma (h)^{-1}B_E)^{c}} |\{a\in A: \Gamma(1,h)\Gamma(x,1)(a,b)\neq \Gamma(\alpha_h(x),1)\Gamma(1,h)(a,b)\}|\\
&\leq \frac{1}{|A||B|}|A| |(B_E\cap\sigma (h)^{-1}B_E)^{c}|
\leq\frac{1}{|B|}(|B\setminus B_E| + |B\setminus \sigma (h)^{-1}B_E|)\\
&\leq 2\kappa, \,\,\,\,\,\,\,\,\text{ {\it(here we use Lemma \ref{kappa})}}
\end{align*}
and this is smaller than $\varepsilon/6$ once we choose $\kappa<\varepsilon/12$ in Lemma \ref{kappa} and use Remark \ref{choice of epsilon prime}. This ends the proof that  $\Gamma$ is $(F_0,\varepsilon,d_{\rm{Hamm}})$-multiplicative.

Let us now prove that $\Gamma$ is $(F_0,\varepsilon,d_{\rm{Hamm}})$-free, namely let us prove that $d_{\text{Hamm}}(\Gamma(x,h),1)>1-\varepsilon$,
 whenever $(x,h)\in F_0\setminus\{1\}$.
Since $\{(a,b):\,\,\sigma(h)b\neq b\} \subseteq \{ (a,b):\Gamma(x,h)(a,b)\neq (a,b)\}$,
it follows that $d_{\text{Hamm}}(\Gamma(x,h),1) \geq d_{\text{Hamm}}(\sigma(h),1)$.
Using that $\sigma$ is  $(E_H,\varepsilon',d_{\text{Hamm}})$-free, for all $h\in E_H\setminus\{1\}$ we have
\begin{equation*}d_{\text{Hamm}}(\Gamma(x,h),1) \geq d_{\text{Hamm}}(\sigma(h),1) \geq 1-\varepsilon' > 1-\kappa >1-\varepsilon,\end{equation*}
where the last inequalities hold because we use Remark \ref{choice of epsilon prime} and we chose $\kappa<\varepsilon/12$ in Lemma \ref{kappa}.
It remains to prove  $(F_0,\varepsilon,d_{\rm{Hamm}})$-freeness in the case when $h=1$. 
To that end,  first observe  that if $(x,1)\in F_0$, then by \eqref{def of E1}  and \eqref{def of E sub H}, $supp(x)\subseteq E$.
Recall that $b\in B_E$, then $\theta_b$ is an injective homomorphism of groups when it is restricted to $\COMM\limits_{E}^{_w}G$.
It follows that 
\begin{equation} \label{theta(x)neq1}
\text{ if }b\in B_E \text { and } (x,1)\in F_0\setminus\{ 1\} \text{, then }\theta_b(x)\neq 1.
\end{equation}
 We can now compute the Hamming distance
\begin{align*}
   d_{\text{Hamm}}(\Gamma(x,1),1) &=\frac{1}{|A||B|}|\{(a,b)\in A\times B: \Gamma(x,1)(a,b)\neq (a,b)\}| \\
&=\frac{1}{|A||B|}|\{(a,b)\in A\times B_E: \varphi(\theta_b(x))(a)\neq a\}|\\
&= \frac{1}{|B|}\sum_{b\in B_E} \frac{1}{|A|}|\{a\in A: \varphi(\theta_b(x))(a)\neq a\}|\\
&=\frac{1}{|B|}\sum_{b\in B_E} d_{\text{Hamm}}(\varphi(\theta_b(x)),1) \\
&\geq \frac{|B_E|}{|B|} (1-\varepsilon') \,\text{ {\it(we use that }}\varphi  \text{ {\it is }} (E_G,\varepsilon',d_{\rm{Hamm}})\text{{\it -free and }}\theta_b(x)\neq 1)\\
 &>(1-\kappa)(1-\varepsilon')\,\,\,\,\,\,\,\,\text{ {\it(here we use Lemma \ref{kappa})}}\\
 &>(1- \kappa)^2> 1-\varepsilon,
 \end{align*}
where the last inequality is because we chose $\kappa<\varepsilon/12$ in Lemma \ref{kappa}.
\end{proof}
\subsection{Weakly sofic verbal wreath products}\label{section weakly sofic}
In \cite{MR2455513} Glebsky and Rivera defined another metric approximation in groups called {\it weakly sofic}. 
In this case, symmetric groups with the Hamming distance get replaced by finite groups with bi-invariant metrics. 
\begin{definition}
A countable discrete group $G$ is said to be weakly sofic if there exists $\alpha>0$ such that for every $\varepsilon>0$ and every $F \subseteq G$ 
finite set, there exist a  finite group $A$ with a bi-invariant metric $d$ and a function $\phi: G \to A$ satisfying that $\phi(1)=1$ and 
\begin{itemize}
\item $(F,\varepsilon,d)$-multiplicative: for all $g,g'\in F$  we have that $d(\phi(g)\phi(g'),\phi(gg')) < \varepsilon$;
 \item $(F,\alpha,d)$-free: if $g\in F\setminus \{1\}$ we have that $d(\phi(g),1) \geq  \alpha$.
\end {itemize} 
\end{definition}

By replacing the metric $d$ with $d^\prime:=\min\{\frac{d}{\alpha},1\}$ it can be assumed that ${\rm diam}(A):=\sup_{x\in A}\{d(x,1)\}=1$ and that 
$d(\phi(g),1)=1$ for all $g\in F\setminus \{1\}$. 

Given a finite group $A$ with a bi-invariant metric $d$ for which ${\rm diam}(A)\leq1$ and a finite set $B$, consider the permutational wreath product $A\wr_B Sym(B)$. In \cite[Proposition 2.9]{HAY18} (see also  \cite[$\S$5]{MR3632893}) it is shown that the following function 
\begin{equation}\label{metric in permut. wreath product}\tilde d(((x_b)_{b\in B},\tau),((y_b)_{b\in B},\rho) ):=
d_{\text{Hamm}}(\tau,\rho)+\frac{1}{|B|}\sum_{\stackrel{b\in B}{\rho(b)=\tau(b)}} d(x_{\tau(b)},y_{\tau(b)})
\end{equation}
is a bi-invariant metric in $A\wr_B Sym(B)$, and with this metric ${\rm diam}(A\wr_B Sym(B))=1$.

\begin{proof}[Sketch of the proof of Theorem \ref{theo hyper intro} in the weakly sofic case] Let $F\subseteq G\wrw H$ be a finite subset and let $\varepsilon >0$. Define the sets
 $F_0,E_1,E_2,E,E_H, B,B_1,B_2,B_E$ the function
$\Theta$ and the set $E_G$ as in the proof of Theorem  $\ref{theo sofic intro}$.

Since $G$ is weakly sofic, the hypothesis in the statement of Theorem \ref{theo hyper intro}, together with Corollary \ref{amenability of many products}, implies that the verbal product $\2W{B }G$ is  weakly sofic. 
Hence, given the finite set $E_G$
and $\varepsilon'>0$,
there exist a finite group $A$ with a bi-invariant metric $d$ of ${\rm diam}(A)=1$, and  $\varphi:\COMM\limits_{B}^{_w}G \to A$ 
a $(E_G,\varepsilon')$-weak sofic approximation of $\2W{B} G$ with $\varphi(1)=1$ and with $d(\varphi(g),1)=1$ for all $g\in E_G\setminus \{1\}$.
Finally, define the function
\begin{align*}
    \Gamma : G \wrw H &\to  A\wr_B Sym(B)    \\
    \Gamma(x,h):=
((y_b)_{b \in B}, \sigma(h)) 
&\text{ where }y_b=
\begin{cases} 
	\varphi(\theta_{b}(x)) &\text{ if } b\in B_E;\\
	1 &\text{ if } b\notin B_E.\\
\end{cases}
\end{align*}
{\bf Claim:} $\Gamma$ is a $(F_0,\varepsilon)$-weakly sofic approximation of $G \wrw H$ 
when $A\wr_B Sym(B)$ is endowed with the bi-invariant metric $\tilde d$ of \eqref{metric in permut. wreath product} and $\alpha=1/2$.
\\
In order to prove that $\Gamma$ is  $(F_0,\varepsilon,\tilde d)$-multiplicative it is enough to show that the premises of Lemma \ref{cuasimultiplicativa} hold true.
In order to check  \ref{cuasimultiplicativa}\ref{item1}, 
\begin{align*}
    \tilde d(\Gamma(x,1)\Gamma(x',1),\Gamma(xx',1))
 &=\frac{1}{|B|}\sum_{b\in B_{E}}d(\varphi( \theta_b(x))\varphi( \theta_b(x')),\varphi( \theta_b(xx')))\\
&=\frac{1}{|B|}\sum_{b\in B_{E}}d(\varphi( \theta_b(x))\varphi( \theta_b(x')),\varphi( \theta_b(x)\theta_b(x')))\\
&\leq\varepsilon^\prime\frac{|B_E|}{|B|}\leq\varepsilon^\prime\leq\kappa\leq\varepsilon/6.
\end{align*}
 The proofs of \ref{cuasimultiplicativa}\ref{item2} and  \ref{cuasimultiplicativa}\ref{item3} are identical to the ones given in the sofic case.
In order to check  \ref{cuasimultiplicativa}\ref{item4}, simple computations give that
\begin{equation*} \Gamma(\alpha_h(x),1)\Gamma(1,h)=((z_b)_{b\in B},\sigma(h)) \text{ where }
 z_b=
\begin{cases} 
	\varphi(\theta_{b}(\alpha_{h}(x))) &\text{ if } b\in B_E;\\
	1 &\text{ if } b\notin B_E;\\
\end{cases}\end{equation*} 
and 
\begin{equation*}\Gamma(1,h) \Gamma(x,1)=((w_{b})_{b\in B},\sigma(h)) \text { where }
w_b=
\begin{cases} 
	\varphi(\theta_{\sigma(h)^{-1}b}(x)) &\text{ if } \sigma(h)^{-1}b\in B_E;\\
	1 &\text{ if } \sigma(h)^{-1}b\notin B_E.\\
\end{cases}\end{equation*} 
 Observe that for $ b\in B_E\cap \sigma(h)^{-1}B_E$, the identity \eqref{igualdad soficidad punto 4} holds true,
 this is because its proof does not depend on the metric approximation of $\2W{B} G$.
 Hence, if $ b\in B_E\cap \sigma(h)B_E$ it follows that $z_b=w_b$.
 then
 \begin{align*}
  \tilde d(\Gamma(1,h)\Gamma(x,1), \Gamma(\alpha_h(x),1)\Gamma(1,h))\}
  &=
  \frac{1}{|B|}\sum_{b\in (B_E\cap\sigma (h)B_E)^{c}} d(z_b,w_b)\\
  &\leq  \frac{|(B_E\cap\sigma (h)B_E)^{c}|}{|B|}
\end{align*}
and the rest of the proof is identical to the sofic case.

It remains to show that $\Gamma$ is  $(F_0,1/2,\tilde d)$-free. If $(x,h)\in F_0$ and $h\neq 1$ then 
$\tilde d(\Gamma(x,h),1)\geq d_{\text{Hamm}}(\sigma(h),1)\geq 1-\varepsilon\geq \frac{1}{2}$, when $\varepsilon<\frac{1}{2}$. 
If $b\in B_E$ and $(x,1)\in F_0$ with $x\neq 1$, then  by \eqref{theta(x)neq1} $\theta_b(x)\neq 1$ and 
\begin{equation*}\tilde d(\Gamma(x,1),1)=\frac{1}{|B|}\sum_{b\in B_{E}}d(\varphi(\theta_b(x),1))=\frac{|B_E|}{|B|}\geq(1-\kappa)\geq\frac{1}{2}.\qedhere
\end{equation*}
\end{proof}

\subsection{Hyperlinear verbal wreath products}\label{section hyperlinear}
The purpose now is to show the hyperlinear case of Theorem \ref{theo hyper intro} from the introduction, which generalizes \cite[Theorem 1.3 (ii)]{HAY18}. 
To that end, recall that
if $\CH$ is a finite dimensional Hilbert space with orthonormal basis $\beta=\{v_1,\dots,v_n\}$, the normalized 
trace on $\mathcal B(\CH)$ is given by the formula \begin{equation*}tr(A):=\frac{1}{\dim(\CH)} \sum_{i=1}^{n} \langle Av_i,v_i\rangle.\end{equation*}
It induces the inner product  $\langle A,B\rangle:=tr(AB^\ast)$  and its corresponding
{\it Hilbert-Schmidt} norm $\|A\|_2:=\langle A,A\rangle^{1/2}$ on $\mathcal B(\CH)$.
We will denote $d_{\text{HS}}(A,B):=\| A-B\|_2$ the Hilbert-Schmidt distance and $\CU(\CH)$ the group of unitary operators on  $\CH$.

A countable discrete group is {\it hyperlinear} if it embeds in  the unitary group of $R^\omega$, the ultrapower of the hyperfinite $\rm{II}_1$ factor. This term was coined by R\u{a}dulescu around 2000, but his article appeared in print several years later, \cite{MR2436761}. R\u{a}dulescu showed that a group is hyperlinear if and only if its group von Neumann algebra embeds in  $R^\omega$, \cite[Poposition 2.6]{MR2436761}
(see also  \cite[Proposition 7.1]{MR2072092}). 
That is exactly the same as saying that hyperlinear groups are the ones whose group von Neumann algebra verifies the Connes'
embedding problem. From this, it is an exercise in ultraproducts and finite von Neumann algebras to show that a group is hyperlinear if and only if it verifies the following ``finitary'' definition.

\begin{definition}
A group $G$ is hyperlinear if for every $\varepsilon>0$ and every $F \subseteq G$ finite set, there exist a finite dimensional Hilbert space $\CH$ and a function $\phi: G \to \CU(\CH)$ satisfying that $\phi(1)=1$ and 
\begin{itemize}
\item $(F,\varepsilon,d_{\rm{HS}})$-multiplicative: for all $g,g'\in F$  we have that $d_{\rm{HS}}(\phi(g)\phi(g'),\phi(gg')) < \varepsilon$;
 \item $(F,\varepsilon)$-trace preserving: if $g\in F\setminus \{1\}$ we have that $|tr(\phi(g))| < \varepsilon$;
\end {itemize} 
\end{definition}

\begin{proof}[Proof of Theorem \ref{theo hyper intro} in the hyperlinear case] \label{prof of hyperlinear}Let $F\subseteq G\wrw H$ be a finite subset and let $\varepsilon >0$. Define the sets
 $F_0,E_1,E_2,E,E_H, B,B_1,B_2,B_E$, the function
$\Theta$ and the set $E_G$ as in the proof of Theorem  $\ref{theo sofic intro}$.

Since $G$ is hyperlinear, the hypothesis in the statement of Theorem \ref{theo hyper intro}, together with Corollary \ref{amenability of many products}, implies that the verbal product $\2W{B }G$ is  hyperlinear. 
Hence, given the finite set $E_G$
and $\varepsilon'>0$,
there exist a finite dimensional Hilbert space $\CH$, with orthonormal basis $\tilde\beta:=\{v_1, \dots, v_n\}$ and  $\varphi:\COMM\limits_{B}^{_w}G \to \CU(\CH)$ 
a $(E_G,\varepsilon')$-hyperlinear approximation of $\2W{B} G$ with $\varphi(1)=1$.
As in the proof of Theorem  $\ref{theo sofic intro}$, we define
\begin{align*}
    \varphi_\star:  \2W{B} G \wr_B Sym(B) &\to \mathcal U(\CH)\wr_B Sym(B) \\
    ((x_b)_{b\in B},\tau) &\mapsto((\varphi(x_b))_{b\in B},\tau).
\end{align*}

Consider the Hilbert space $\bigoplus_B \CH$, and its orthonormal basis $\beta:=\{v_i^b: b\in B\}$, where, for each $b\in B$, $v_i^b$ denotes the element $v_i$ of $\tilde\beta$ in the $b^{th}$ copy of $\CH$ inside $\bigoplus_B \CH$. 
 We denote with $\rho$ the  permutational action of $Sym (B)$ on  $\bigoplus_B \CU(\CH)$ and $\rho_\tau$ the automorphism corresponding to $\tau\in Sym(B)$.

For $U\in \CU(\CH)$, we will denote $U^b\in \bigoplus_B \CU(\CH)$ the image of $U$ under the embedding of $\CU(\CH)$ into the $b^{th}$-coordinate of $\bigoplus_B \CU(\CH)$ and let $\eta$ be the diagonal embedding defined by
\begin{align*}
   \eta: \bigoplus_{B} \mathcal{U}(\CH) &\to \CU\Big(\bigoplus_B \CH\Big)\\
   \Big\langle\eta ( (U_b)_{b\in B})v_i^{\tilde b},v_j^{\tilde{\tilde b}}\Big\rangle &:=
   \begin{cases} 
	   \langle U_{\tilde b} v_i^{\tilde b},v_j^{\tilde b}\rangle &\text{ if } \tilde b=\tilde{\tilde b};\\
  			0 &\text{ if not}.
 \end{cases}
\end{align*}

The permutation by blocks of the basis $\beta$ defines the homomorphism
\begin{align*}
    P: Sym(B) &\to \mathcal U\Big(\bigoplus_B \CH\Big)\\
P(\tau)(v_i^b)&:=  v_i^{\tau(b)}.   
\end{align*}
Note that for $U\in\CU(\CH)$,  and for any $b\in B$, we have that $P(\tau)\eta(U^b)P(\tau^{-1})= \eta(\rho_\tau(U^b))$. Indeed, this is because
 for every $v_i^{\tilde b} \in \beta$ we have, on the one hand
\begin{equation*}\eta(\rho_\tau(U^b))v_i^{\tilde{b}} =\eta(U^{\tau(b)})v_i^{\tilde{b}}= \begin{cases}
U^{\tau(b)}v_i^{\tau(b)} &\text{if }  \tilde{b}=\tau({b}); \\
v_i^{\tilde{b}} &\text{if not};
\end{cases}\end{equation*}
and on the other hand, 
\begin{equation*}
P(\tau)\eta(U^b)P(\tau^{-1})v_i^{\tilde b}=
P(\tau)\eta(U^b)v_i^{\tau^{-1}(\tilde b)}
=
\begin{cases}
P(\tau)U^b v_i^{\tau^{-1}(\tilde b)} &\text{ if } b=\tau^{-1}(\tilde b);\\
P(\tau)v_i^{\tau^{-1}(\tilde b)} &\text{ if not}.
\end{cases}
\end{equation*}
By taking  projections and inclusions, it follows that for any $(U_b)_{b\in B}\in \bigoplus_B \CU(\CH)$ we have that 
\begin{equation*}P(\tau)\eta((U_b)_{b\in B})P(\tau^{-1})= \eta(\rho_\tau((U_b)_{b\in B})).\end{equation*}
Using this identity, it is easy to show that the following function is a group homomorphism 
\begin{align*}
    \psi:\mathcal U(\CH) \wr_B Sym(B) &\to \mathcal U\Big(\bigoplus_B \CH\Big) \\
\psi((U_b)_{b\in B},\tau)&:= \eta((U_b)_{b\in B}) P(\tau).
\end{align*}
 
Define $\Gamma:G\wrw H\to \mathcal U\big(\bigoplus_B \CH\big)$ as the composition map $\Gamma:= \psi \varphi_\star \Theta$.
To be more explicit, for any $(x,h) \in G\wrw H$, we have that
\begin{align} \label{Gamma explicit hyperlinear}
    \Gamma(x,h)(v_i^b)= \begin{cases} v_i^{\sigma(h)b} &\text{ if } \sigma(h)b \notin B_E;\\
    \eta \big((\varphi(\theta_{\sigma(h)b}(x)))^{\sigma(h)b}\big) v_i^{\sigma(h)b} &\text{ if } \sigma(h)b \in B_E.
    \end{cases}
\end{align}

{\bf Claim:} $\Gamma$ is a $(F_0,\varepsilon)$-hyperlinear approximation of $G \wrw H$.\\
In order to prove that $\Gamma$ is  $(F_0,\varepsilon,d_{\rm{HS}})$-multiplicative it is enough to show that the premises of Lemma \ref{cuasimultiplicativa} hold true.\\
In order to check  \ref{cuasimultiplicativa}\ref{item1}, 
take $(x,1), (x',1)$ with $x,x' \in E_1$. Note that 
\begin{equation*}
supp(x),supp(x'),supp(xx')\subseteq supp(E_1)\subseteq E.
\end{equation*}
In particular $x,x',xx'\in \2W{E}G$.
On the one hand, if $b\notin B_E$,  by \eqref{Gamma explicit hyperlinear}, we have that $\Gamma(x,1)(v_i^b)=v_i^b$ for any $x\in\2W{H}G$.
On the other hand, if $b \in B_E$, note that $\Gamma(x^\prime,1)(v_i^b)$ is in the $b^{th}$-copy of $\CH$ inside $\bigoplus_B \CH$ and hence we have the identities
\begin{equation} \label{identity used to prove item 1}   \Gamma(x,1)\Gamma(x',1)v_i^b=\Gamma(x,1)\Big(\eta\big((\varphi(\theta_b(x')))^{b}\big)v_i^b\Big)= 
\eta\Big(\big(\varphi(\theta_b(x))\varphi(\theta_b(x'))\big)^b\Big)v_i^b;
\end{equation}
and
\begin{equation} \label{identity used to prove item 1b} 
\Gamma(xx',1)v_i^b=\eta\Big(\big(\varphi(\theta_b(xx'))\big)^b\Big)v_i^b.
\end{equation}
Then 
\begin{align*}
 \|\Gamma(x,1)&\Gamma(x',1) -  \Gamma(xx',1)\|_2^2 
  =\frac{1}{n|B|} \sum\limits_{b\in B}\sum\limits_{i=1}^{n}\big\|\Gamma(x,1)\Gamma(x',1)v_i^{b} - \Gamma(xx',1)v_i^{b} \big\|^2\\
   &=\frac{1}{|B|}\sum\limits_{b\in B_E}\frac{1}{n}\sum\limits_{i=1}^{n}\big\|\Gamma(x,1)\Gamma(x',1)v_i^{b} - \Gamma(xx',1)v_i^{b}\big \|^2\\\
 &= \frac{1}{|B|}\sum\limits_{b\in B_E}\frac{1}{n}\sum\limits_{i=1}^{n}\big\|\eta\Big(\big(\varphi(\theta_b(x))\varphi(\theta_b(x'))\big)^b\Big)v_i^b
 - \eta\Big(\big(\varphi(\theta_b(xx'))\big)^b\Big)v_i^b\big \|^2\\
 &=\frac{1}{|B|}\sum\limits_{b\in B_E} \big\| \varphi(\theta_b(x))\varphi(\theta_b(x'))
  -  \varphi(\theta_b(xx'))\big\|_2^2 \\
 &= \frac{1}{|B|}\sum\limits_{b\in B_E}\big\|\varphi(\theta_b(x))\varphi(\theta_b(x')) - \varphi(\theta_b(x)\theta_b(x'))\big\|_2^2\,\,,
    \end{align*}
    where the last equality is valid because  $\theta_b$ is a group homomorphism in $\2W{E}G$ and  $x,x'\in \2W{E}G$.
Since $\varphi$ is $(E_G,\varepsilon',d_{\text{HS}})$-multiplicative and since, by definition, $\theta_b(x),\theta_b(x')\in E_G$,
it follows that
\begin{equation*}d_{\text{HS}}(\Gamma(x,1)\Gamma(x',1),\Gamma(xx',1))<\varepsilon'\sqrt{\frac{|B_E|}{|B|}}<\varepsilon'<\kappa<\varepsilon/6,\end{equation*}
once we choose $\kappa<\varepsilon/6$ in Lemma \ref{kappa} and use Remark \ref{choice of epsilon prime}.

In order to check  \ref{cuasimultiplicativa}\ref{item2}, observe first that for any $h\in H$,  \eqref{Gamma explicit hyperlinear} gives
\begin{equation} \label{identity used to prove item 2}
 \Gamma(1,h')\Gamma(1,h)v_i^b=v_i^{\sigma(h')\sigma(h)b}
\,\,\,\,\,\,\text{ and }\,\,\,\,\, 
\Gamma(1,hh')v_i^{b}=v_i^{\sigma(hh')b}.
\end{equation}
Now take $h,h' \in E_H\supseteq E_2$. Then
\begin{align*}
\big\|\Gamma(1,h)\Gamma(1,h')-\Gamma(1,hh')\big\|_2^2 &= \frac{1}{n|B|} \sum_{b\in B} \sum\limits_{i=1}^{n} \big\|v_i^{\sigma(h)\sigma(h')b} - v_i^{\sigma(hh')b}\big\|^2 \\
&= \frac{1}{n|B|} \sum_{b\in B }\sum\limits_{i=1}^{n} 2\big |\{b\in B : \sigma(h)\sigma(h')b \neq \sigma(hh')b \}\big|\\
&= 2 d_{\text{Hamm}}(\sigma(h)\sigma(h'),\sigma(hh')).
\end{align*}
It follows that 
$d_{\text{HS}}(\Gamma(1,h)\Gamma(1,h'),\Gamma(1,hh')) \leq \sqrt{2\varepsilon'} \leq \sqrt{2\kappa}\leq \varepsilon/6$
once we choose $\kappa<\varepsilon^2/72$ in Lemma \ref{kappa} and use Remark \ref{choice of epsilon prime}.
 
In order to check  \ref{cuasimultiplicativa}\ref{item3}, a straightforward  computation using \eqref{Gamma explicit hyperlinear} shows that for any $(x,h)\in G \wrw H$ it holds that 
$\Gamma(x,h)=\Gamma(x,1) \Gamma(1,h)$.

In order to check  \ref{cuasimultiplicativa}\ref{item4}, using \eqref{Gamma explicit hyperlinear}, we have that
\begin{equation*}\Gamma(1,h) \Gamma(x,1)v_i^{b}= \begin{cases}
 \Gamma(1,h)\eta\big((\varphi(\theta_b(x)))^b\big)v_i^{b}=\eta \big(\varphi(\theta_b(x))^{\sigma(h)b}\big)v_i^{\sigma(h)b} &\text{if }b\in B_E;\\
v_i^b &\text{if }b\notin B_E. \end{cases}
 \end{equation*}
Observe that for $ b\in B_E\cap \sigma(h)^{-1}B_E$, the identity \eqref{igualdad soficidad punto 4} holds true,
 this is because its proof does not depend on the sofic or hyperlinear approximation  of $\2W{B} G$. Hence, if $ b\in B_E\cap \sigma(h)^{-1}B_E$
 then
 \begin{align}\label{identity used to prove item 4}
 \Gamma(\alpha_h(x),1)\Gamma(1,h)v_i^b&= \Gamma(\alpha_h (x),1)v_i^{\sigma(h)b}
 = \eta\Big(\big(\varphi(\theta_{\sigma(h)b}(\alpha_h(x)))\big)^{\sigma(h)b}\Big)v_i^{\sigma(h)b}\\
 &= \eta\Big(\big(\varphi(\theta_b(x))\big)^{\sigma(h)b}\Big)v_i^{\sigma(h)b}
 =\Gamma(1,h) \Gamma(x,1)v_i^{b}.\nonumber
\end{align}
 With this at hand, we now proceed to estimate the Hilbert-Schmidt distance in  \ref{cuasimultiplicativa}\ref{item4}.
\begin{align*}
 \big\|\Gamma(1,h)& \Gamma(x,1) - \Gamma(\alpha_h(x),1)\Gamma(1,h)\big\|_2^2\\
&=\frac{1}{n|B|}\sum_{b\in B}\sum\limits_{i=1}^{n}\big\|\Gamma(1,h)\Gamma(x,1)v_i^b - \Gamma(\alpha_h(x),1)\Gamma(1,h)v_i^b\big\|^2\\ 
&
=\frac{1}{n|B|}\sum_{b\in (B_E\cap\sigma (h)^{-1}B_E)^{c}}\sum\limits_{i=1}^{n} \big\|\Gamma(1,h)\Gamma(x,1)v_i^b - \Gamma(\alpha_h(x),1)\Gamma(1,h)v_i^b\big\|^2\\
&\leq \frac{4}{|B|} |(B_E\cap\sigma (h)^{-1}B_E)^{c}|
\leq\frac{4}{|B|}\Big(|B\setminus B_E| + |B\setminus \sigma (h)^{-1}B_E|\Big)
\leq 8\kappa< (\varepsilon/6)^2,
\end{align*}
once we choose $\kappa<\varepsilon^2/288$ in Lemma \ref{kappa} and use Remark \ref{choice of epsilon prime}.

Let us now prove that $\Gamma$ is $(F_0,\varepsilon)$-trace preserving. Let $(x,h)\in F_0\setminus\{1\}$.
Suppose first that $h \neq 1$ and note that if $b\in B_E$ then for all $1 \leq i \leq n$ we have that $\Gamma(x,h)s_i^b$ belongs in the subspace generated by $\{s_i^{\sigma(h)b}: 1\leq i\leq n\}$ which is orthogonal to $s_i^b$. So in this case $\langle \Gamma(x,h)s_i^b , s_i^b\rangle = 0$. Then
\begin{align*}
|{\rm{tr}}(\Gamma(x,h))| 
&=|\langle \Gamma(x,h), 1\rangle| \\
&\leq \frac{1}{n|B|} \sum_{b\in B} \sum_{i=1}^n |\langle \Gamma(x,h)s_i^{b}, s_i^b\rangle| \\
&=\frac{1}{n|B|}\Bigg( \sum_{b\in B_E} \sum_{i= 1}^n |\langle \Gamma(x,h) s_i^{b}, s_i^b\rangle| + \sum_{b\in B\setminus B_E} \sum_{i= 1}^n |\langle \Gamma(x,h) s_i^{b}, s_i^b\rangle|\Bigg) \\
&=  \frac{1}{n|B|} \sum_{b\in B\setminus B_E} \sum_{i= 1}^n| \langle \Gamma(x,h) s_i^{b}, s_i^b\rangle|
\leq \frac{| B\setminus B_E|}{|B|} 
\leq \kappa < \varepsilon,
\end{align*}
where the last inequality is because we chose $\kappa<\varepsilon^2/288$ in Lemma \ref{kappa}.\\
If $b\in B_E$ and $(x,1)\in F_0$ with $x\neq 1$, then  by \eqref{theta(x)neq1} $\theta_b(x)\neq 1$, so
\begin{align*}
 |{\rm{tr}}(\Gamma(x,1))| &=|\langle \Gamma(x,1),1\rangle| \\
& \leq \frac{1}{|B|} \sum_{b\in B}\Big |\frac{1}{n}\sum\limits_{i=1}^n \langle \Gamma(x,1)s_i^b,s_i^b\rangle\Big |\\
&= \frac{1}{|B|}\sum_{b\in B_E}\Big |\frac{1}{n}\sum\limits_{i=1}^n \langle \Gamma(x,1)s_i^b,s_i^b\rangle\Big |+
\frac{1}{|B|}\sum_{b\in B \setminus B_E}
\Big |\frac{1}{n}\sum\limits_{i=1}^n \langle \Gamma(x,1)s_i^b,s_i^b\rangle\Big|\\
&\leq \frac{1}{|B|} \sum_{b\in B_E}\Big |\frac{1}{n}\sum\limits_{i=1}^n 
\big\langle \eta\Big(\big(\varphi(\theta_b(x))\big)^b\Big) s_i^b,s_i^b\big\rangle\Big| + \frac{|B \setminus B_E|}{|B|} \\
&\leq \frac{1}{|B|}\sum_{b \in B_E} |{\rm{tr}}(\varphi(\theta_b(x)))| + \kappa
\leq \frac{|B_E|}{|B|}\varepsilon' + \kappa 
 < 2\kappa< \varepsilon,
\end{align*}

where the last inequality is because we chose $\kappa<\varepsilon^2/288$ in Lemma \ref{kappa}.
\end{proof}

\subsection{Linear sofic  verbal wreath products}

If $K$ is a field, then $d_{\rm{rk}}(A,B):=\frac{1}{n}\rm{rank}(A-B)$ is a bi-invariant metric in $\rm{GL}_{n}(K)$.
In \cite{MR3592512}, Arzhantseva and P\u{a}unescu used this metric to introduce another approximation property in groups called linear soficity. 
As usual, this property is defined in terms of embeddings in ultraproducts,
 this time of invertible matrices with entries in a fixed field $K$. In \cite[Proposition 5.13]{MR3592512} 
 the authors give an equivalent finitary definition that we proceed to record.
\begin{definition}
 Let  $K$ be a field. A group $G$ is linear sofic over $K$ if for every $\varepsilon>0$ and every 
 $F \subseteq G$ finite set, there exist  a function $\phi: G \to \rm{GL}_n(K)$ satisfying that $\phi(1)=1$ and 
\begin{itemize}
\item $(F,\varepsilon,d_{\rm{rk}})$-multiplicative: for all $g,g'\in F$  we have that $d_{\rm{rk}}(\phi(g)\phi(g'),\phi(gg')) < \varepsilon$;
 \item $(F,\varepsilon,d_{\rm{rk}})$-free: if $g\in F\setminus \{1\}$ we have that $d_{\rm{rk}}(1,\phi(g))>\frac{1}{4} -\varepsilon$.
\end {itemize} 
\end{definition}
Minor modifications to the proof for the hyperlinear case of Theorem \ref{theo hyper intro} allow us to prove the variant for linear sofic groups.
\begin{proof}[Sketch of the proof of Theorem \ref{theo hyper intro} in the linear sofic case]
Let $K$ be a field. In the proof of the hyperlinear case of Theorem  \ref{theo hyper intro} given in $\S$\ref{section hyperlinear}, replace $\CH$ by $K^n$, $\CU(\CH)$ by $\rm{GL}_n(K)$,
the orthonormal basis  $\tilde \beta$ of $\CH$ by a basis  $\tilde \beta$ of $K^n$,
$\varphi$ by a $(E_G,\varepsilon^\prime)$-linear sofic approximation  of $\2W{B} G$,
and the inner product used to define the diagonal embedding $\eta$ by taking coordinates with respect to the basis $\beta$. 
Then  \eqref{Gamma explicit hyperlinear}, \eqref{identity used to prove item 1}, \eqref{identity used to prove item 1b}, \eqref{identity used to prove item 2}, and \eqref{identity used to prove item 4}
remain valid because they do not depend on the matrices being unitary. 
With these identities at hand, it is easy to see that $\Gamma$ verifies the four items of Lemma \ref{cuasimultiplicativa}. Thus  $\Gamma$ is $(F_0,\varepsilon,d_{\rm{rk}})$-multiplicative. 

Let us prove that
 $d_{\rm{rk}}(\Gamma(x,h),1)\geq \frac{1}{4}-\varepsilon$,
whenever $(x,h)\in F_0\setminus\{1\}$.
Suppose first that $h \neq 1$ and note that $\ker(\Gamma(x,h)-1)\subseteq \text{span}\{v_i^b\in\beta:\sigma(h)b=b\}$. This subspace has dimension
equal to $n|\{b\in B:\sigma(h)b=b\} |=n(|B|-|B|d_{\rm{Hamm}}(\sigma(h),1))\leq n|B|\varepsilon^{\prime}$. 
It follows that $\frac{1}{n|B|}{\rm{rank}}(\Gamma(x,h)-1)\geq 1-\varepsilon^{\prime}\geq 1-\varepsilon$.
When $h=1$, and $b\in B_E$, then $\theta_b(x)\neq 1$.
 We then have that 
\begin{align*}
\frac{1}{n|B|}{\rm rank}(\Gamma(x,1)-1)&=\frac{1}{|B|}\sum_{b\in B_{E}}\frac{1}{n}{{\rm rank}}(\varphi(\theta_b(x))-1)
\geq\frac{|B_{E}|}{|B|}\Big(\frac{1}{4}-\varepsilon^{\prime}\Big)\\
&\geq(1-\kappa)\Big(\frac{1}{4}-\varepsilon^{\prime}\Big)>\frac{1}{4}-\varepsilon,
\end{align*}
once we choose an adequate $\kappa$ in Lemma \ref{kappa}.
\end{proof}
We end this section by proving the application of the Shmelkin embedding discussed in the introduction.
\begin{proof}[Proof of Corollary \ref{coro Shmelkin}]
The Shmelkin embedding \cite{MR0193131} is 
\begin{equation}\label{Shmelkin1}
\faktor{F}{W(G)}\hookrightarrow \Big(\faktor{F}{W(F)}\Big)\wrw \faktor{F}{G}\,\,.
\end{equation}
Observe that the Magnus embedding is the case when $W=\{n_1\}=\{[x_2,x_1]\}$.

If $W=\{n_k\}$, $\faktor{F}{W(F)}$ is a free nilpotent group of order $k$; if  $W=\{s_k\}$, $\faktor{F}{W(F)}$ is a free solvable group of derived length $k$, and if $W=\{x_1^k\}$ with $k=2,3,4,6$, $\faktor{F}{W(F)}$ is a finitely generated free Burnside group of exponent $k=2,3,4,6$, hence all of them are amenable, and thus sofic. Assume that the quotient group $\faktor{F}{G}$ is sofic. Then, by Theorem \ref{theo sofic intro}, the restricted verbal wreath products in \eqref{Shmelkin1}  are sofic. Since soficity is preserved by subgroups, it follows that $\faktor{F}{n_k(G)}$, $\faktor{F}{s_k(G)}$ and $\faktor{F}{G^k}$ are sofic. 
The proof in the case when  $\faktor{F}{G}$ has the Haagerup property is almost identical.
\end{proof}

\textbf{Acknowledgments.} J. Brude was supported in part by a CONICET Doctoral Fellowship. R. Sasyk was supported in part through the grant PIP-CONICET 11220130100073CO. We thank Prof. Denis Osin for pointing us that verbal wreath products of groups were already introduced by Shmelkin in the mid sixties and for sharing with us some relevant references. We also thank Pedro Marun for helping us to get hold of several Math Reviews unavailable in our country.

\bibliographystyle{plain}
\bibliography{verbal}

\end{document}